\documentclass[10pt,letterpaper,showkeys]{article}
\usepackage{amsmath}
\usepackage{amssymb}
\usepackage{authblk}
\usepackage{graphicx,subfigure}
\usepackage{graphics}
\usepackage{epsfig}
\usepackage{color}

\newtheorem{theorem}{Theorem}[section]
\newtheorem{lemma}[theorem]{Lemma}
\newtheorem{corollary}[theorem]{Corollary}

\newtheorem{remark}[theorem]{Remark}
\newtheorem{example}[theorem]{Example}
\newtheorem{assumption}{Assumption}
\newtheorem{definition}{Definition}

\newcommand{\beq}{\begin{equation}}
\newcommand{\eeq}{\end{equation}}
\newcommand{\beqa}{\begin{eqnarray}}
\newcommand{\eeqa}{\end{eqnarray}}
\newcommand{\beqas}{\begin{eqnarray*}}
\newcommand{\eeqas}{\end{eqnarray*}}
\newcommand{\ba}{\begin{array}}
\newcommand{\ea}{\end{array}}
\newcommand{\bi}{\begin{itemize}}
\newcommand{\ei}{\end{itemize}}
\newcommand{\gap}{\hspace*{2em}}

\newcommand{\nn}{\nonumber}

\def\Argmin{{\rm Argmin}}

\def\vgap{\vspace*{.1in}}
\def\QED{\ifhmode\unskip\nobreak\fi\ifmmode\ifinner\else\hskip5pt\fi\fi
  \hbox{\hskip5pt\vrule width5pt height5pt depth1.5pt\hskip1pt}}
  \def\bL{{\bar L}}
\def\bt{{\bar t}}
\def\cB{{\cal B}}
\def\cI{{\cal I}}
\def\cJ{{\cal J}}
\def\cK{{\cal K}}
\def\cL{{\cal L}}

\def\cS{{\cal S}}
\def\dist{{\rm dist}}
\def\eps{{\epsilon}}
\def\feas{{\rm feas}}

\def\init{{\rm init}}
\def\M{\mathcal{M}}
\def\A{{A}}
\def\avg{{\rm avg}}
\def\b{{b}}
\def\c{{\rm c}}
\def\cC{{\cal C}}
\def\cS{{\cal S}}
\def\co{{\rm conv}}
\def\cU{{\cal U}}
\def\range{{\cal R}}
\def\rank{{\rm rank}}
\def\tL{{\widetilde \cL}}

\def\tw{{\tilde w}}
\def\tx{{\tilde x}}
\def\ut{{\underline t}}

\title{Sparse Recovery via Partial Regularization: \\ Models, Theory and Algorithms
\thanks{This work was supported in part by NSERC Discovery Grant.}
}
\author{Zhaosong Lu}
\author{Xiaorui Li}
\affil{Department of Mathematics, Simon Fraser University, Canada \\ 
\{zhaosong,xla97\}@sfu.ca}

\date{November 23, 2015}

\begin{document}

\maketitle

\begin{abstract}
In the context of sparse recovery, it is known that most of existing regularizers 
such as $\ell_1$ suffer from some bias incurred by some leading entries (in magnitude) 
of the associated vector. To neutralize this bias, we propose a class of models with partial regularizers for recovering a sparse solution of a linear system. We show that 
every local minimizer of these models is sufficiently sparse or the magnitude of all its 
nonzero entries is above a uniform constant depending only on the data 
of the linear system. Moreover, for a class of partial regularizers, any global minimizer 
of these models is a sparsest solution to the linear system. We also establish some 
sufficient conditions for local or global recovery of the sparsest solution to the linear 
system, among which one of the conditions is weaker than the best known restricted 
isometry property (RIP) condition for sparse recovery by $\ell_1$. In addition, a 
first-order feasible augmented Lagrangian (FAL) method is proposed for solving these 
models, in which each subproblem is solved by a nonmonotone proximal gradient 
(NPG) method. Despite the complication of the partial regularizers, we show that 
each proximal subproblem in NPG can be solved as a certain number of one-dimensional 
optimization problems, which usually have a closed-form solution. We also show that 
any accumulation point of the sequence generated by FAL is a first-order stationary point 
of the models. Numerical results on compressed sensing and sparse logistic regression 
demonstrate that the proposed models substantially outperform the widely used ones in 
the literature in terms of solution quality.                  

\vskip14pt


\end{abstract}

\section{Introduction} 
\label{intro}

Over the last decade, seeking a sparse solution to a linear system has attracted a great deal of attention in science and engineering (e.g., see  \cite{Lustig2007,Baraniuk2007,Lustig2008,Herman2009,Tropp2010,Davenport2012}). It has found numerous applications in imaging processing, signal processing and statistics. 
Mathematically, it can be formulated as the $\ell_0$ minimization problem 
\beq \label{L0}
\min_{x \in \Re^n} \{ \| x \|_0: \|A x-b\| \le \sigma\},
\eeq
where $A \in \Re^{m \times n}$ is a data matrix, $b \in \Re^m$ is an observation vector, $\sigma \ge 0$ is the noisy level (i.e., $\sigma=0$ for noiseless data and observation and  $\sigma>0$ for noisy data and/or observation), and $\| x \|_0$ denotes the cardinality of the vector $x$.

 Given that $\| \cdot \|_0$ is an integer-valued, discontinuous and nonconvex function, it is generally hard to solve \eqref{L0}. One common approach in the literature (e.g., see \cite{Ti96,ChDoSa98,Candes2005}) is to solve the following $\ell_1$ relaxation problem  
\beq \label{L1}
\min_{x \in \Re^n} \{ \| x \|_1: \|A x-b\| \le \sigma\}.
\eeq
Under some condition on $A$ such as restricted isometry property (RIP) and  
null space property (NSP), it has been shown in the literature that the solution of \eqref{L0} 
can be stably recovered by that of \eqref{L1} (e.g., see \cite{Candes2005,Don06,DoElTe06,Cai2013,Cai2014} and the references therein). It is well-known that  $\|\cdot\|_1$ is generally a biased approximation to $\|\cdot\|_0$. To neutralize the bias in $\|x\|_1$ incurred by some leading entries (in magnitude) of $x$, one can replace $\|x\|_1$ by $\sum^n_{i=r+1} |x|_{[i]}$ in \eqref{L1} and 
obtain a less biased relaxation for \eqref{L0}:    
\beq \label{partial-L1}
\min_{x \in \Re^n} \left\{ \sum^n_{i=r+1} |x|_{[i]}: \|A x-b\| \le \sigma\right\},
\eeq
where $r\in\{0,1,\ldots, n-1\}$ and $|x|_{[i]}$ is the $i$th largest element in $\{|x_1|, |x_2|, \ldots, |x_n|\}$. 
We now provide a simple example to illustrate this. 
\begin{example}
The general solution to $Ax=b$ is $x=(t,t,1-t, 2-t,3-t)^T$ for all $t \in \Re$, where 
\[
A =\left(
\begin{matrix}
1&-1&0&0&0\\
1&0&1&0&0\\
1&0&0&1&0\\
1&0&0&0&1
\end{matrix}
\right ), \quad \quad
b=\left(
\begin{matrix}
0\\
1\\
2\\
3
\end{matrix}
\right ).
\]
Thus $x^{\rm s}=(0,0,1,2,3)$ is the sparsest solution to $Ax=b$. One can verify that $x^{\rm f}=(1,1,0,1,2)^T$ is the unique 
solution of \eqref{L1} with $\sigma=0$, while $x^{\rm p}= (0,0,1,2,3)^T$ is the unique solution of \eqref{partial-L1} with $\sigma=0$ and $r=2$ or $3$. Clearly, the partial $\ell_1$ model \eqref{partial-L1} recovers successfully the sparest solution  but the full $\ell_1$ model \eqref{L1} fails.  
\end{example}

A more general relaxation of \eqref{L0} considered in the literature is in the form of 
\beq \label{phi-L1}
\min_{x \in \Re^n} \left\{ \sum^n_{i=1} \phi(|x_i|): \|A x-b\| \le \sigma\right\},
\eeq
where $\phi$ is some regularizer (e.g., see \cite{FrFr93,Fu98,FaLi01,WeElScTi03,Ch07,CaWaBo08,Ch08,Zhang09,CHZ10}).  Some popular $\phi$'s are listed as follows.
\bi
\item[(i)]  ($\ell_1$ \cite{Ti96,ChDoSa98,Candes2005}):  $\phi(t) =  t$ \ $\forall t \ge 0$;
\item[(ii)] ($\ell_q$ \cite{FrFr93,Fu98}): $\phi(t)= t^q$ \ $\forall t \ge 0$;
\item[(iii)] (Log \cite{WeElScTi03,CaWaBo08}): $\phi(t) = \log(t+\varepsilon)- \log(\varepsilon)$ \ $\forall t \ge 0$;
\item[(iv)] (Capped-$\ell_1$ \cite{Zhang09}): $\phi(t) = \left\{
\ba{ll}
t & \mbox{if} \ 0 \le t < \nu, \\
\nu & \mbox{if} \ t \ge \nu;
\ea\right.$
\item[(v)]  (MCP \cite{CHZ10}): $\phi(t) = \left\{
\ba{ll}
\lambda t -\frac{t^2}{2\alpha} &  \mbox{if} \ 0 \le t < \lambda \alpha, \\ [4pt]
\frac{\lambda^2 \alpha}{2} & \mbox{if } \ t \ge  \lambda \alpha;
\ea\right.
$
\item[(vi)]  (SCAD \cite{FaLi01}): $\phi(t) = \left\{
\ba{ll}
\lambda t &  \mbox{if} \  0 \le t \le \lambda, \\ [4pt]
\frac{-t^2+2\beta\lambda t -\lambda^2}{2(\beta-1)} &  \mbox{if} \ \lambda < t <  \lambda \beta, \\ [4pt]
\frac{(\beta+1)\lambda^2}{2} & \mbox{if } \ t \ge  \lambda \beta,
\ea\right.
$
\ei 
 where $0<q<1$, $\varepsilon>0$, $\nu>0$, $\lambda>0$, $\alpha >1$ and  $\beta>1$ are some parameters. One can observe that problem \eqref{L1} is a special case of \eqref{phi-L1} with $\phi(t)=t$ for $t\ge 0$. Analogous to $\|\cdot\|_1$, $\sum^n_{i=1} \phi(|x_i|)$ 
is generally also a biased approximation to $\|\cdot\|_0$ in some extent.  
Similar as above, one can neutralize the bias in $\sum^n_{i=1} \phi(|x_i|)$ incurred by some leading entries (in magnitude) of $x$ to obtain a less biased relaxation for \eqref{L0}:    
\beq \label{pL1}
\min_{x \in \Re^n} \left\{ \sum^n_{i=r+1} \phi(|x|_{[i]}): \|A x-b\| \le \sigma\right\}
\eeq
for some $r\in\{0,1,\ldots, n-1\}$. In this paper we are interested in problem \eqref{pL1} 
for which $\phi$ satisfies the following assumption.
\begin{assumption} \label{assump-phi}
 $\phi$ is lower semi-continuous and increasing in $[0,\infty)$. Moreover,  
$\phi(0)=0$.
\end{assumption}

Under Assumption \ref{assump-phi}, we study some theoretical properties of problem \eqref{pL1} and also develop an efficient algorithm for solving it. In particular,  we show that problem \eqref{pL1} has at least an 
optimal solution. We also establish some sparsity inducing properties of \eqref{pL1} with $\sigma=0$. It is shown that any local minimizer of \eqref{pL1} has at most $r$ 
nonzero entries or the magnitude of all its nonzero entries is not too small, which is indeed above a uniform constant depending only on $A$ and $b$. Moreover, we prove that for a class of $\phi$, any global minimizer of \eqref{pL1} is a sparsest solution to $Ax=b$. In addition, we study some sufficient conditions for local or global recovery of a sparest solution $x^*$ of $Ax=b$ by model \eqref{pL1}. It is shown that if a local null space property holds at $x^*$ for $A$ and $\phi$, $x^*$ is a {\it strictly local} minimizer of \eqref{pL1} with $\sigma=0$. Specifically, for $\phi(t)=|t|^q$ with $q\in(0,1]$, we develop some weaker 
RIP conditions for local recovery than the well-known optimal ones for \eqref{L1}. It is also shown that if a global null space property holds at $x^*$ for $A$ and $\phi$, $x^*$ is a {\it unique global} minimizer of \eqref{pL1} with $\sigma=0$. Some RIP conditions are particularly derived for global recovery by \eqref{pL1} with $\sigma=0$ and $\phi(t)=|t|^q$ with $q\in(0,1]$. Furthermore, we study stable recoverability of \eqref{pL1} with $\phi(t)=|t|$ and $\sigma \ge 0$ under some RIP conditions. Additionally, we develop a first-order feasible augmented Lagrangian (FAL) method for solving problem 
\eqref{pL1} in which each subproblem is solved by a nonmonotone proximal gradient (NPG)  method. Though the subproblems of NPG appear to be sophisticated due to $\sum^n_{i=r+1} \phi(|x|_{[i]})$, we are fortunate to show that they 
can be solved as $n-r$ number of one-dimensional optimization problems that  
have a closed-form solution for the widely used $\phi$. We also show that any accumulation point of the sequence generated by FAL is a first-order stationary point 
of problem \eqref{pL1}. Finally we conduct some numerical experiments to compare the performance of \eqref{pL1} with \eqref{phi-L1}. Numerical results demonstrate that \eqref{pL1} substantially outperforms \eqref{phi-L1} in terms of solution quality.          
      
The rest of this paper is organized as follows. 
In Section \ref{existence} we establish the existence of optimal solutions for problem \eqref{pL1}. Some sparsity inducing properties for \eqref{pL1} are studied in Section \ref{sparse-prop}. We establish some sufficient conditions for local or global recovery of a sparse solution by problem \eqref{pL1} and study its stable recoverability for in Section \ref{recovery}. We propose an efficient algorithm for solving \eqref{pL1} in 
Section \ref{alg}. Numerical results are presented in Section \ref{results}. Finally we 
make some concluding remarks in Section \ref{remarks}.   

\subsection{Notation and terminology}
\label{notation}

In this paper the set of all nonnegative real numbers is denoted by $\Re_+$.  For a real number $t$, let $t_+ = \max\{0,t\}$. Let $\Re_+^n$ denote the nonnegative orthant of $\Re^n$. For any $x,y \in \Re^n$, $x \geq y$ means $x - y \in \Re_+^n$. 
$\{0,1\}^n$ denotes the set of all $n$-dimensional vectors with binary entries. For any $x\in\Re^n$,  $\|x\|_0$ and $\|x\|$ denote the cardinality (i.e., the number of nonzero entries) and the Euclidean norm of $x$, respectively, and $\|x\|_q =\sqrt[\leftroot{-3}\uproot{3}q]{\sum\limits_{i=1}^n |x_i|^q}$ for any $q \in (0,1]$. In addition, $x_{[i]}$ denotes the $i$th largest entry of $x$ for $i = 1, \ldots, n$ and $|x|$ stands for the $n$-dimensional vector whose $i$th entry is $|x_{i}|$. Given an integer $1 \leq k \leq n$, we denote by $x_{-\max (k)}$ the vector resulted from $x$ by replacing the entries of $x$ with $k$ largest absolute value by zero, and set $x_{\max (k)} = x - x_{-\max (k)}$. A vector $x$ is said to be $k$-sparse if $\|x\|_0 \le k$. Given any $A \in \Re^{m \times n}$,  $\rank(A)$ stands for the rank of $A$, and the null space and range space of $A$ are denoted by $\mathcal{N}(A)$ and $\range(A)$, respectively. Given an index set $J\subseteq \{ 1, \ldots, n \}$, $|J|$ and $J^\c$ denote the size of $J$ and the complement of $J$ in $\{ 1, \ldots, n \}$, respectively. For any $x \in \Re^n$ and $M \in \Re^{m \times n}$, $x_J$ denotes the subvector formed by the entries of $x$ indexed by $J$, and $M_J$ denotes the submatrix formed by the columns of $M$ indexed by $J$. For any closed set $S$, let ${\rm dist}(x,S)$ denote the distance from $x$ to $S$, and ${\rm conv}(S)$ denote the convex hull of $S$.

We recall from \cite[Definition~8.3]{Rock98} that for a proper 
lower semi-continuous function $\vartheta$ (not necessarily locally Lipschitz), the limiting and horizon subdifferentials are defined respectively as
\begin{equation*}
\begin{split}
\partial \vartheta(x)&:=\left\{v\,\left|\;\exists x^k \stackrel{\vartheta}{\to} x,\;v^k\to v\;\mbox{ with }\liminf_{z \to x^k}\frac{\vartheta(z)-\vartheta(x^k)-\langle v^k,z-x^k\rangle}{\|z-x^k\|}\ge 0\quad \forall k\right.\right\},\\
\partial^{^\infty}\!\!\vartheta(x)&:=\left\{v \,\left|\;\exists x^k \stackrel{\vartheta}{\to} x,\;\lambda_k v^k\to v, \lambda_k\downarrow 0\;\mbox{ with }\liminf_{z \to x^k}\frac{\vartheta(z)-\vartheta(x^k)-\langle v^k,z-x^k\rangle}{\|z-x^k\|}\ge 0\quad \forall k\right.\right\},
\end{split}
\end{equation*}
where $\lambda_k\downarrow 0$ means $\lambda_k > 0$ and $\lambda_k\to 0$, and $x^k  \stackrel{\vartheta}{\to} x$ means both $x^k \to x$ and $\vartheta(x^k)\to \vartheta(x)$. The following concept was introduced in \cite{Candes2005} by Cand\`{e}s and Tao.  
\begin{definition} [restricted isometry constant]
Given $\A \in \Re^{m \times n}$ and an integer $1\leq k \leq n$, the restricted isometry constant (RIC) of order $k$, denoted by $\delta^A_k$,  is the smallest number $\delta$ such that for all $k$-sparse vectors $x \in \Re^n$,
\[
(1-\delta)\| x \|^2 \leq \|\A x\|^2 \leq (1+\delta)\|x\|^2.
\]
\end{definition}
We define $\delta^A_k$ as $\delta^A_{\left \lceil k \right\rceil}$ when $k>0$ is not an integer. 

\section{Existence of optimal solutions}
\label{existence}

In this section we establish the existence of optimal solutions of problem \eqref{pL1}. To this end, we first consider a more general class of problems in the form of 
\beq \label{p1}
F^* = \inf_{x \in \Re^n} \left \{F(x):= g(\A x-\b) + \Psi(x) \right \},
\eeq
where  $\A \in \Re^{m \times n}$, $\b \in \Re^m$, and $F^*$, $g$ and $\Psi$ satisfy the following assumption.

\begin{assumption} \label{assump-exist}
\bi
\item[(a)] $g: \Re^m \to \Re \cup \{\infty\}$ is lower semi-continuous in $\Re^n$ 
and any level set of $g$ is bounded.
\item[(b)] $\Psi: \Re^n \to \Re$ is bounded below,  lower semi-continuous, and moreover, 
\[
\Psi(y) \ge \Psi(z) \ \  \mbox{if}  \ \ |y| \ge |z|.
\] 
\item[(c)] $F^*$ is finite.
\ei
\end{assumption}

We next show that under Assumption \ref{assump-exist}, the infimum 
of problem \eqref{p1} is attainable at some point.

\begin{theorem} \label{th1}
Under Assumption \ref{assump-exist}, the infimum $F^*$ of problem \eqref{p1} is achievable at some point, that is, \eqref{p1} has at least an 
optimal solution. 
\end{theorem}

\begin{proof}
By the definition of $F^*$, there exists a sequence $\{x^k\}$ such that 
\beq \label{F-seq}
\lim\limits_{k \to \infty}  F(x^k) =  \lim\limits_{k \to \infty} g(\A x^k - \b) + \Psi(x^k) = F^*.
\eeq
By Assumption \ref{assump-exist}(b) and \ref{assump-exist}(c), we know that $\{\Psi(x^k)\}$ is bounded below and $F^*$ is finite. These together with \eqref{F-seq} imply that $\{g(\A x^k - \b)\}$ is 
bounded above. Using this and Assumption \ref{assump-exist}(a), we conclude that 
$\{Ax^k - b\}$ is bounded and so is $\{Ax^k\}$. Let 
$J = \{j: \{x_j^k\} {\rm{\ is\ bounded}} \}$. It then follows that $\{A_Jx^k_J\}$ is bounded. This together with the boundedness of $\{Ax^k\}$ implies that $\{A_{J^\c}x^k_{J^\c}\}$ is bounded. 
Consider the linear system 
\[
A_{J^\c} z= A_{J^\c}x^k_{J^\c}.
\]
Clearly it has at least a solution $z=x^k_{J^\c}$.  By Hoffman Lemma \cite{Hoffman1952},  there 
exist $z^k$ and some constant $\zeta>0$ depending only on $A_{J^\c}$ such that  
\beq \label{hoffman}
A_{J^\c} z^k= A_{J^\c}x^k_{J^\c}, \quad \|z^k\| \le \zeta \|A_{J^\c}x^k_{J^\c}\|.
\eeq
Let $\tx^k=(x^k_J, z^k)$. Using \eqref{hoffman}, the boundedness of 
 $\{A_{J^\c}x^k_{J^\c}\}$, 
and the definitions of $J$ and $\tx^k$, one can observe that $A \tx^k = Ax^k$ and $\{\tx^k\}$ is 
bounded. Considering a convergent subsequence if necessary, assume for convenience that 
$\tx^k \to x^*$ for some $x^*$. By the definition of $F$ and the lower semi-continuity of $g$ and $\Psi$, we can see that $F$ is lower semi-continuous and hence 
\beq \label{F-limit}
\liminf_{k\to \infty} F(\tx^k) \ge F(x^*). 
\eeq  
 In addition, by the definition of $J$, we know that $\{x^k_{J^\c}\}$ is 
unbounded. This along with the boundedness of $\{\tx^k_{J^\c}\}$ implies $|x_{J^\c}^k| \geq |\tx_{J^\c}^k|$ for sufficiently large $k$. It follows from this and $\tx_J^k = x_J^k$ that $|x^k| \ge |\tx^k|$ for sufficiently large $k$. This and Assumption \ref{assump-exist}(b) implies that $\Psi(x^k) \geq \Psi(\tx^k)$ for sufficiently large $k$. By this relation, $A\tx^k=Ax^k$ and \eqref{p1}, we have that for sufficiently large $k$,
\[
F(x^k) = g(Ax^k-b)+\Psi(x^k) \ge g(A\tx^k-b)+\Psi(\tx^k) = F(\tx^k).
\]
In view of this, \eqref{F-seq} and \eqref{F-limit}, one has 
\[
F^* = \lim\limits_{k \to \infty}  F(x^k) \ge \liminf_{k\to \infty} F(\tx^k) \ge F(x^*).
\]
This along with the definition of $F^*$ implies $F(x^*)=F^*$ and hence $x^*$ is an optimal solution of problem \eqref{p1}.
\end{proof}

\gap 

As a consequence of the above theorem, we next establish the existence of optimal solutions of problem \eqref{pL1}. 

\begin{theorem} \label{exist-soln}
Suppose that Assumption \ref{assump-phi} holds for $\phi$ and the feasible region of  \eqref{pL1} is nonempty. Then problem \eqref{pL1} has at least an optimal solution.  
\end{theorem}

\begin{proof}
Let $\Omega=\{w\in\Re^m: \|w\| \le \sigma\}$ and $\iota_\Omega$ be the indicator function of $\Omega$, that is, $\iota_\Omega(w)=0$ if  $w \in \Omega$ and $\infty$ 
otherwise. One can observe that \eqref{pL1} is a special case of \eqref{p1} with 
$g=\iota_\Omega$ and $\Psi(\cdot)=\sum^n_{i=r+1} \phi(|\cdot|_{[i]})$. Since $\phi$ 
satisfies Assumption \ref{assump-phi}, it is not hard to verify that  Assumption \ref{assump-exist} holds for such $g$ and $\Psi$. The conclusion then follows from Theorem \ref{th1}. 
\end{proof}

\gap

The following model has been widely used in the literature for 
sparse recovery:
\beq
\min\limits_{x\in\Re^n} \ \lambda \|Ax-b\|^2 + \Psi(x), \label{mod1}\\ 
\eeq
 where $\lambda>0$ and $\Psi$ is a sparsity-inducing regularizer 
such as $\ell_0$ \cite{BlDa08,BlDa09,LuZh13,Lu14},  $\ell_1$ \cite{Ti96,ChDoSa98,Candes2005}), $\ell_q$  \cite{FrFr93,Fu98}, 
Capped-$\ell_1$ \cite{Zhang09}, Log \cite{WeElScTi03,CaWaBo08}, MCP \cite{CHZ10} and SCAD \cite{FaLi01} that are presented in Section \ref{intro}.  It is not hard to observe that \eqref{mod1} is a special case of 
\eqref{p1} with $g(\cdot) = \lambda\|\cdot\|^2_2$ for \eqref{mod1}.  
The following result thus immediately follows from  Theorem \ref{th1}. 
 
\begin{corollary}
Problem \eqref{mod1} has at least an optimal solution if $\Psi$ is  one of the regularizers $\ell_0$, $\ell_1$,  $\ell_q$,   Capped-$\ell_1$, Log,  MCP and SCAD.
\end{corollary}


\section{Sparsity inducing property}
\label{sparse-prop}

In this section we study some sparsity inducing properties of model \eqref{pL1} with $\sigma=0$, that is, 
\beq\label{p2}
\min\limits_{x\in\Re^n}\left\{\sum\limits_{i=r+1}^n \phi (|x|_{[i]}): 
A x -b=0\right\}.
\eeq
In particular, we show that any local minimizer of \eqref{p2} is $r$-sparse or the magnitude of all its nonzero entries cannot be too small, which is indeed above a uniform constant depending only on $A$ and $b$.  We also show that under some assumption on $\phi$, any global minimizer of \eqref{p2} is a sparsest solution to the linear system $\A x = \b$. 

Before proceeding, we introduce some notations that will be used subsequently.  Given any $x\in\Re^n$ and $j \in\{1,\ldots,n\}$, we define 
\beq \label{index}
\ba{ll}
\cI_0(x) = \{ i: x_i=0 \},\quad  & \cI^\c_0(x) = \{ i: x_i \neq 0 \}, \\ [5pt]
\cI_j^=(x) = \{ i: |x_i| = |x|_{[j]} \}, \quad & \cI^>_j(x) = \{ i: |x_i| > |x|_{[j]} \}, \\ [5pt]
 \cI^<_j(x) = \{ i: |x_i| < |x|_{[j]} \}, \quad & \cI^{<+}_j(x) = \{ i: 0<|x_i| < |x|_{[j]} \}, \\ [5pt]
\cI^\le_j(x) = \{ i: |x_i| \leq |x|_{[j]} \}, \quad &  \cI^{\le+}_j(x) = \{ i: 0 < |x_i| \le |x|_{[j]} \}. 
\ea
\eeq
 It is not hard to observe that 
\[
\ba{ll}
 \cI^>_j(x) \cup \cI^=_j(x) \cup \cI^<_j(x) =\{1,\ldots, n\}, &
 I^\c_0(x)=I^{\le+}_j(x) \cup I^>_j(x), \\ [6pt]
 \cI^\le_j(x) =  \cI^=_j(x) \cup \cI^<_j(x), & 
 \cI^\le_j(x) =  \cI_0(x) \cup \cI^{\le+}_j(x).
\ea
\]

\vgap

The following lemma characterizes some nonzero components of a local minimizer of problem \eqref{p2}.
  
\begin{lemma} \label{lem:lowerbnd}
Assume $\phi''(t)<0$ for all $t>0$. Suppose $x^*$ is a local minimizer of problem \eqref{p2}. Let 
$\cI^>_{r+1}=\cI^>_{r+1}(x^*)$, $\cI^{\le+}_{r+1}=\cI^{\le+}_{r+1}(x^*)$, $\cJ$ a subset of $\cI^>_{r+1}$ 
such that $\rank(\A_\cJ) = \rank(\A_{\cI^>_{r+1}})$ and $\A_\cJ$ has full column rank, and 
$B = [A_{\cI^{\le+}_{r+1}} \ A_\cJ] $, where $\cI^>_{r+1}(x^*)$ and $\cI^{\le+}_{r+1}(x^*)$ are defined in 
\eqref{index}. Then $x^*$ is $r$-sparse or the following statements hold:
\bi
\item[(i)] $B$ has full column rank;
\item[(ii)] $x^*_{\cI^{\le+}_{r+1}} = [ e_1 \cdots e_l]^T (B^T B)^{-1} B^T \b$, where $l = |\cI^{\le+}_{r+1}|$ and $e_i$ is the $i$th coordinate vector for all $i$.
\end{itemize}
\end{lemma}

\begin{proof}
If $x^*$ is a $r$-sparse solution to $Ax=b$, the conclusion clearly holds. We now suppose that $x^*$ is not a $r$-sparse solution to $Ax=b$. It implies that  $\cI^{\le+}_{r+1} \neq \emptyset$ and 
$|x^*_i|>0$ for all $i\in \cI^=_{r+1}(x^*)$. We next show that statements (i) and (ii) hold. 

(i) One can observe from \eqref{index} that when 
$x$ is sufficiently close to $x^*$, 
\beq \label{index-relation}
\cI^<_{r+1}(x^*) \subseteq \cI^<_{r+1}(x), \quad  \cI^>_{r+1}(x^*) \subseteq \cI^>_{r+1}(x).
\eeq
Let $t_1 = |\cI^=_{r+1}(x^*)|$, $t_2=n-r- |\cI^<_{r+1}(x^*)|$ and 
\[
\quad \cS = \left\{\cK: \cK \subseteq \cI^=_{r+1}(x^*), \ |\cK|=t_2\right\}.
\]
Observe that $t_1,\ t_2>0$ and $\cS \neq \emptyset$. 
Using \eqref{index-relation} and the definitions of $t_1$, $t_2$ and $\cS$,  we can observe that when $x$ is sufficiently close 
to $x^*$, 
\beq
\sum\limits_{i=r+1}^n \phi (|x|_{[i]}) = \sum\limits_{i\in \cI^<_{r+1}(x^*)} \phi (|x_i|) +\min\limits_{\cK \in \cS} \sum\limits_{i\in\cK}  \phi (|x_i|) 
 \le  \sum\limits_{i\in \cI^<_{r+1}(x^*)} \phi (|x_i|) + \frac{t_2}{t_1} \sum\limits_{i\in \cI^=_{r+1}(x^*)}  \phi (|x_i|). \label{phi-relation1}
\eeq
Notice that 
\beq \label{phi-relation2}
\sum\limits_{i=r+1}^n \phi (|x^*|_{[i]}) = \sum\limits_{i\in \cI^<_{r+1}(x^*)} \phi (|x^*_i|) + \frac{t_2}{t_1} \sum\limits_{i\in \cI^=_{r+1}(x^*)}  \phi (|x^*_i|).
\eeq
Since $x^*$ is a local minimizer of problem \eqref{p2},  we have 
$\sum\limits_{i=r+1}^n \phi (|x|_{[i]}) \ge \sum\limits_{i=r+1}^n \phi (|x^*|_{[i]}) $
when $x$ is sufficiently close to $x^*$. By this relation, \eqref{phi-relation1} and \eqref{phi-relation2}, one can see that when $x$ is sufficiently 
close to $x^*$, 
\[
\sum\limits_{i\in \cI^<_{r+1}(x^*)} \phi (|x_i|) + \frac{t_2}{t_1} \sum\limits_{i\in \cI^=_{r+1}(x^*)}  \phi (|x_i|) \ge \sum\limits_{i\in \cI^<_{r+1}(x^*)} \phi (|x^*_i|) + \frac{t_2}{t_1} \sum\limits_{i\in \cI^=_{r+1}(x^*)}  \phi (|x^*_i|).
\]
Hence, $x^*$ is a local minimizer of the problem 
\[
\min\limits_x \left\{ \sum\limits_{i\in \cI^<_{r+1}(x^*)} \phi (|x_i|) + \frac{t_2}{t_1} \sum\limits_{i\in \cI^=_{r+1}(x^*)}  \phi (|x_i|): Ax=b\right\}.
\]
It then follows that $(x^*_{\cI^{<+}_{r+1}},x^*_{\cI^=_{r+1}}, x^*_{\cI^>_{r+1}})$ is a local minimizer 
of the problem 
\beq \label{loc-prob}
\ba{ll}
\min & \sum\limits_{i\in \cI^{<+}_{r+1}} \phi (|x_i|) + \frac{t_2}{t_1} \sum\limits_{i\in \cI^=_{r+1}}  \phi (|x_i|) \\
&  A_{\cI^{<+}_{r+1}}x_{\cI^{<+}_{r+1}}+ A_{\cI^{=}_{r+1}}x_{\cI^{=}_{r+1}}+ A_{\cI^{>}_{r+1}}x_{\cI^{>}_{r+1}}=b,  
\ea
\eeq
where $\cI^{<+}_{r+1} =\cI^<_{r+1}(x^*)$, $\cI^=_{r+1}=\cI^=_{r+1}(x^*)$ 
and $\cI^>_{r+1}=\cI^>_{r+1}(x^*)$. By the second-order necessary optimality condition of  \eqref{loc-prob}, 
one has
\beq \label{2nd-opt1}
\sum\limits_{i \in \cI^{<+}_{r+1}} \phi '' (|x^*_i|) h_i^2 + \frac{t_2}{t_1} \sum\limits_{i \in \cI_{r+1}^=} \phi '' (|x^*_i|) h_i^2  \geq 0
\eeq
for all $h=(h_{\cI^{<+}_{r+1}},h_{\cI^=_{r+1}}, h_{\cI^>_{r+1}})$ satisfying 
\beq \label{2nd-opt2}
 A_{\cI^{<+}_{r+1}}h_{\cI^{<+}_{r+1}}+ A_{\cI^{=}_{r+1}}h_{\cI^{=}_{r+1}}+ A_{\cI^{>}_{r+1}}h_{\cI^{>}_{r+1}}=0.
\eeq
Recall that $|x^*_i|>0$ for all $i\in\cI^=_{r+1}$. Hence, $\cI^{\le+}_{r+1} = \cI^{<+}_{r+1} \cup \cI^=_{r+1}$.  This together with the 
assumption $\phi''(t) < 0$ for all $t > 0$ implies that $\phi''(|x^*_i|)<0$ for all $i\in \cI^{\le+}_{r+1}$. Using this relation, $t_1,t_2>0$, $\cI^{\le+}_{r+1} = \cI^{<+}_{r+1} \cup \cI^=_{r+1}$, 
\eqref{2nd-opt1} and \eqref{2nd-opt2}, we have 
\beq \label{2nd-opt3}
A_{\cI^{\le+}_{r+1}}h_{\cI^{\le+}_{r+1}}+ A_{\cI^{>}_{r+1}}h_{\cI^{>}_{r+1}}=0 \ \Longrightarrow \ h_{\cI^{\le+}_{r+1}}=0.
\eeq 
Recall that $\cJ \subseteq \cI^>_{r+1}$ is such that $\rank(\A_\cJ) = \rank(\A_{\cI^>_{r+1}})$ 
and $\A_\cJ$ has full column rank. This along with \eqref{2nd-opt3} yields
\[
A_{\cI^{\le+}_{r+1}}h_{\cI^{\le+}_{r+1}}+ A_\cJ h_\cJ=0 \ \Longrightarrow \ h_{\cJ\cup\cI^{\le+}_{r+1}}=0.
\]
It implies that $B= [A_{\cI^{\le+}_{r+1}} \  A_\cJ]$ has full column rank. 

(\romannumeral2) Notice that $Ax^*=b$. It then follows that 
\beq \label{feas}
A_{\cI^{\le+}_{r+1}} x^*_{\cI^{\le+}_{r+1}} + A_{\cI^{>}_{r+1}} x^*_{\cI^{>}_{r+1}} = \b.
\eeq
Since $\cJ \subseteq \cI^>_{r+1}$ and $\rank(\A_\cJ) = \rank(\A_{\cI^>_{r+1}})$, there exists 
some $\tx^*\Re^{|J|}$ such that $A_\cJ \tx^* = A_{\cI^{>}_{r+1}} x^*_{\cI^{>}_{r+1}}$. This together 
with \eqref{feas} yields
  \beq \nn
A_{\cI^{\le+}_{r+1}} x^*_{\cI^{\le+}_{r+1}} + A_\cJ \tx^*= \b.
\eeq
In view of this and the fact that $B= [A_{\cI^{\le+}_{r+1}} \  A_\cJ]$ has full column rank, one has 
\beq \nn
\left(\ba{c}
 x^*_{\cI^{\le+}_{r+1}} \\ [8pt] 
\tilde{x}^*
\ea \right) = (B^T B)^{-1} B^T \b,
\eeq
which immediately implies that statement (ii) holds.
\end{proof}

\gap

As a consequence of the above lemma, we next show that any local minimizer of \eqref{p2} is $r$-sparse or the magnitude of all its nonzero entries cannot be too small, which is indeed above a uniform constant depending only on $A$ and $b$. 

\begin{theorem} \label{thm:lowerbnd} 
Assume $\phi''(t)<0$ for all $t>0$. Let 
\beqa
\cB &:=& \left\{\cI \subseteq \{1,\ldots,n\}: A_\cI \ \mbox{has full column rank and} \  (A_\cI)^T b \neq 0\right\}, \nn \\ [5pt] 
\delta_\cI &:=& \min\left\{\Big|[(A_\cI)^T A_\cI]^{-1}(A_\cI)^T b\Big|_i: \ba{c} 
\left|[(A_\cI)^T A_\cI]^{-1}(A_\cI)^T b\right|_i>0, \nn \\ [5pt]
1\le i \le |\cI|
\ea\right\} \quad\quad \forall \cI \in \cB, \nn
\\ [5pt]
\delta &:=& \min\{\delta_\cI: \cI \in \cB\}. \label{delta}
\eeqa
Let $x^*$ be an arbitrary local minimizer of problem \eqref{p2} and 
\beq \label{l0-opt}
K=\min\left \{ \|x\|_0: \A x = \b \right \}.
\eeq
 There hold:
\bi
\item[(i)] If $0 \le r <K$, $|x^*_i| \ge \delta$ for all $i\in\cI^\c_0(x^*)$;
\item[(i)] If $K \le r \le n-1$,  $x^*$ is $r$-sparse  or $|x^*_i| \ge \delta$ for all $i\in\cI^\c_0(x^*)$.
\ei
\end{theorem}

\begin{proof}
 Notice $\cI^\c_0(x^*) = \cI^{\le+}_{r+1}(x^*) \cup \cI^>_{r+1}(x^*)$. 
In view of this, Lemma \ref{lem:lowerbnd} and the definition of $\delta$, 
one can observe that if $x^*$ is  non-$r$-sparse, then $\cI^{\le+}_{r+1}(x^*) \neq \emptyset$ and 
\[
|x^*_i| \ge \min\{|x^*_j|: j \in \cI^{\le+}_{r+1}(x^*)\} \ge \delta \quad\quad \forall i \in \cI^\c_0(x^*).
\]
In addition, by  the definition of $K$, one can see that if $0 \le r <K$, $x^*$ must be non-$r$-sparse. The conclusion of 
this theorem then immediately follows from these observations and Lemma 
\ref{lem:lowerbnd}.
\end{proof}

\gap

{\bf Remark:} The quantity $\delta$ given in \eqref{delta}  is independent 
to $r$, but only depends on $A$ and $b$. In addition, the above theorem implies that the components of a non-$r$-sparse local minimizer of \eqref{p2} with magnitude below $\delta$ must be zero. Therefore, any  local minimizer of \eqref{p2} is either $r$-sparse or the magnitude of 
its nonzero entries is above $\delta$.

\gap

In the context of sparse recovery, the regularizers often associate with some parameter 
that controls the proximity to the $\ell_0$ regularizer. For example, the well-known $\ell_q$  
regularizer is one of them. We next consider a special class of model \eqref{p2} with 
$\phi(\cdot) = \psi(\cdot, q)$ for some parameter $q$, namely,        
\beq
\label{p3}
\begin{aligned}
& \underset{x \in \Re^n}{\text{min}}
& & \sum\limits_{i=r+1}^n \psi(|x|_{[i]},q) \\
& \text{s.t.}
& & \A x = \b,
\end{aligned}
\eeq
where $\psi: \Re_+ \times (0,\alpha) \to \Re_+$ for some $\alpha \in (0,\infty]$ satisfies the following assumption. 

\begin{assumption} \label{assump-psi} 
\bi
\item[(a)] $\psi(\cdot,q)$ is lower semi-continuous and increasing in $\Re_+$ for any $q\in(0,\alpha)$;
\item[(b)] $\frac{\partial^2 \psi}{\partial t^2} (t,q) < 0$ for all $t>0$ and $q\in (0,\alpha)$;
\item[(c)] $\psi(0,q) = 0$ and $\lim\limits_{q \to 0^+} \psi(t,q) = 1$ for all $t>0$.
\ei 
\end{assumption}

 One can observe from Theorem \ref{th1}  that under Assumption \ref{assump-psi}(a), problem \eqref{p3} has at least  an optimal solution. 
In addition, under Assumption \ref{assump-psi}(b), Theorem \ref{thm:lowerbnd} also holds for problem\eqref{p3}. We next show that when $q$ is sufficiently small, any global minimizer of \eqref{p3} is a sparsest solution to the system $Ax=b$. This indicates that problem \eqref{p3} is capable of inducing a sparse solution.

\begin{theorem} \label{th3}
Suppose that Assumption \ref{assump-psi} holds for $\psi$. Let $K$ be defined in \eqref{l0-opt}. 
Then for any $0 \le r \leq \min\{K,n-1\}$ and sufficiently small $q > 0$, any global minimizer 
$x^*$ of problem \eqref{p3} is also that of problem \eqref{l0-opt}, that is, $x^*$ is a sparsest solution to the system $Ax=b$. 
 \end{theorem}

\begin{proof} 
Let $r\in [0, \min\{K,n-1\}]$ be arbitrarily chosen, $\tx^*$ and $x^*$ a global minimizer of \eqref{l0-opt} and \eqref{p3}, respectively,  and let  $\delta$ be given in \eqref{delta}. In view of Assumption 
\ref{assump-psi}(c), $\|\tx^*\|_0=K$ and $r\le K$, one can observe that 
\[
\lim\limits_{q\to 0^+} \frac{\sum\limits_{i=r+1}^n \psi(|\tx^*|_{[i]},q)}{\psi(\delta,q)} = K-r.
\]
Hence, there exists some $0<\eps<\alpha$ such that $\psi(\delta,q)>0$ and 
\beq \label{ratio-ineq}
\frac{\sum\limits_{i=r+1}^n \psi(|\tx^*|_{[i]},q)}{\psi(\delta,q)}  \le K-r + \frac12 \quad\quad \forall q \in (0,\eps).
\eeq
By Theorem \ref{thm:lowerbnd}, we know that $x^*$ is a $r$-sparse 
solution to $Ax=b$ or $|x^*_i| \ge \delta$ for all $i\in\cI^\c_0(x^*)$. 
We next show that $x^*$ is an optimal solution of \eqref{l0-opt} if 
$q\in(0,\eps)$ by considering two cases.

Case 1): $x^*$ is a $r$-sparse solution to $Ax=b$. This together with the assumption $r\le K$ implies 
that $x^*$ is a $K$-sparse solution to $Ax=b$. It then follows from \eqref{l0-opt} that the conclusion 
holds.

Case 2): $|x^*_i| \ge \delta$ for all $i\in\cI^\c_0(x^*)$. This together with the monotonicity of $\psi(\cdot,q)$ implies that $\psi(|x^*_i|,q) \ge \psi(\delta,q)$. Since $A\tx^*=b$ and $x^*$ is a global minimizer of  problem \eqref{p3}, one has  
\[
\sum\limits_{i=r+1}^n \psi(|x^*|_{[i]},q) \leq \sum\limits_{i=r+1}^n \psi(|\tx^*|_{[i]},q).
\]
Using these two relations and $\psi(0,q)=0$,  we obtain that
\beq \nn
\left(n-r-|\cI_0(x^*)|\right) \psi(\delta,q) \le \sum\limits_{i=r+1}^n \psi(|x^*|_{[i]},q) \leq \sum\limits_{i=r+1}^n \psi(|\tx^*|_{[i]},q) \quad\quad \forall q \in (0,\eps),
\eeq 
which together with \eqref{ratio-ineq} and $\psi(\delta,q)>0$ implies $n-r-|\cI_0(x^*)| \le K-r+1/2$ and hence 
$|\cI_0(x^*)| \ge n-K-1/2$. Since $|\cI_0(x^*)|$ is integer, it then follows that $|\cI_0(x^*)| \ge n-K$ and hence $|\cI^\c_0(x^*)| \le K$. This along with \eqref{l0-opt} implies 
that $x^*$ is an optimal solution of \eqref{l0-opt} if $q\in(0,\epsilon)$.
\end{proof}

\section{Sparse recovery}
\label{recovery}

In this section we first establish some sufficient conditions for local or global recovery of a sparsest solution to $Ax=b$ by model \eqref{pL1} with $\sigma=0$, namely, \eqref{p2}. Then we study the stable recoverability of model \eqref{pL1} with $\phi(\cdot)=|\cdot|$.

\subsection{Local sparse recovery by model \eqref{p2}}

In this subsection we study some sufficient conditions for local recovery of 
a sparsest solution $x^*$ to the system $Ax=b$ by model \eqref{p2}. In particular, we show that if a local null space property holds, $x^*$ is a  
strictly local minimizer of \eqref{p2}. We also establish a local recovery result 
for \eqref{p2} under some RIP conditions and the assumption $\phi(\cdot)=|\cdot|^q$ for $q\in(0,1]$. To this end, we first introduce the definition of 
local null space property.

\begin{definition}[local null space property]
The {\it local null space property} {\rm{LNSP}}($r$,$\phi$) holds 
for the system $Ax=b$ at a solution $x^*$ with $\|x^*\|_0 \ge r$ if there exists some $\epsilon > 0$ such that 
\beq \label{lnsp}
\sum\limits_{i\in \cI_0(x^*)} \phi(|h_i|)-\max\limits_{J\in \cJ} \sum\limits_{i\in J \cup \cI_{r+1}^{<+}(x^*)} \phi(|h_i|) > 0
\eeq
for all $h \in \mathcal{N}(\A)$ with $0<\|h\|< \epsilon$, 
where 
\beq \label{cJ}
\cJ = \left\{J \subseteq \cI_{r+1}^=(x^*): |J|=\|x^*\|_0-r-|\cI_{r+1}^{<+}(x^*)|\right\}.
\eeq 
\end{definition}

\vgap

We next show that for a class of $\phi$, a solution $x^*$ of $\A x = \b$ is a strictly local minimizer of \eqref{p2} if the local null space property holds at $x^*$.

\begin{theorem} \label{th4} 
Suppose that Assumption \ref{assump-phi} holds for $\phi$ and additionally   
\beq \label{phi1}
\phi(s) \geq \phi(t)-\phi(|s-t|) \quad\quad \forall s, t \geq 0.
\eeq
Let $x^* \in \Re^n$ be such that $\A x^* = \b$. Assume that  {\rm{LNSP}}($r$,$\phi$) holds at $x^*$ for 
some $0\le r \le \|x^*\|_0$. Then $x^*$ is a strictly local minimizer of problem \eqref{p2} with such $\phi$ and $r$.
\end{theorem}

\begin{proof}
One can observe that $\cI^<_{r+1}(x^*) \subseteq \cI^<_{r+1}(x^*+h)$  
and $\cI^>_{r+1}(x^*) \subseteq \cI^>_{r+1}(x^*+h)$ for sufficiently 
small $h$. Hence, for every sufficiently small $h$, there exists some $J \subseteq \cI_{r+1}^=(x^*)$ (dependent on $h$) with $|J| = \|x^*\|_0-r-|\cI_{r+1}^{<+}(x^*)|$ such that
\beq \label{sum-phi}
\sum\limits_{i=r+1}^n \phi(|x^*+h|_{[i]}) = \sum\limits_{i\in J \cup \cI_{r+1}^{<+}(x^*)} \phi(|x^*_i+h_i|) + \sum\limits_{i \in \cI_0(x^*)}\phi(|h_i|).
\eeq
By $|x^*_i+h_i| \ge \big||x^*_i|-|h_i|\big|$,  
the monotonicity of $\phi$ and  \eqref{phi1}, one has 
\beq \label{phi-tineq}
\phi(|x^*_i+h_i|) \ge \phi\left(\big||x^*_i|-|h_i|\big|\right)  \ge \phi(|x^*_i|)-\phi(|h_i|)\quad \forall i.
\eeq 
Using this relation and \eqref{sum-phi}, we obtain that
\beq \label{sum-phi1}
\sum\limits_{i=r+1}^n \phi(|x^*+h|_{[i]}) 
\geq \sum\limits_{i \in J \cup \cI_{r+1}^{<+}(x^*)} \phi(|x^*_i|) - \sum\limits_{i \in J \cup \cI_{r+1}^{<+}(x^*)} \phi(|h_i|) + \sum\limits_{i \in \cI_0(x^*)}\phi(|h_i|).
\eeq
Since $\phi(0)=0$, one has $\phi(x^*_i)=0$ for all $i\in \cI_0(x^*)$. By this and the definition of $J$, we have
\[
\sum\limits_{i \in J \cup \cI_{r+1}^{<+}(x^*)} \phi(|x^*_i|) = 
\sum\limits_{i \in J \cup \cI_{r+1}^{<+}(x^*)} \phi(|x^*_i|) + \sum\limits_{i \in \cI_0(x^*)}\phi(|x^*_i|) = \sum\limits_{i=r+1}^n \phi(|x^*|_{[i]}). 
\]
This relation and \eqref{sum-phi1} yield
\[
\sum\limits_{i=r+1}^n \phi(|x^*+h|_{[i]}) \ge \sum\limits_{i=r+1}^n \phi(|x^*|_{[i]}) - \sum\limits_{i \in J \cup \cI_{r+1}^{<+}(x^*)} \phi(|h_i|) + \sum\limits_{i \in \cI_0(x^*)}\phi(|h_i|).
\]
By this and \eqref{lnsp},  we see that for all sufficiently small 
$h \in \mathcal{N}(\A) \backslash \{ 0 \}$, there holds
\beq \nn
\sum\limits_{i=r+1}^n \phi(|x^*+h|_{[i]}) > \sum\limits_{i=r+1}^n \phi(|x^*|_{[i]}).
\eeq
Hence, $x^*$ is a strictly local minimizer of problem \eqref{p2}. 
\end{proof}

\gap

{\bf Remark:} It is not hard to see that $\phi(t)=|t|^q$ for some $q\in(0,1]$ satisfies the assumption stated in Theorem \ref{th4}. 

\vgap

As an immediate consequence,  we obtain that under some suitable 
assumption, a sparsest solution $x^*$ to the system $Ax=b$ is a 
strictly local minimizer of \eqref{p2}. One implication of this result is 
that an optimization method applied to \eqref{p2} likely converges  
to $x^*$ if its initial point is near $x^*$.

\begin{corollary}[local sparse recovery] \label{K-loc-recovery}
Suppose that Assumption \ref{assump-phi} and  \eqref {phi1} hold for $\phi$,   
and $x^*$ is a sparsest solution to the system $Ax=b$. Assume that the local null space property {\rm{LNSP}}($r$,$\phi$) holds at $x^*$ for some $0\le r\le K$, where $K$ is defined in \eqref{l0-opt}.  Then $x^*$ is a strictly local minimizer of problem \eqref{p2} with such $\phi$ and $r$.
\end{corollary}


The above local sparse recovery is established based on the local null space property {\rm{LNSP}}($r$,$\phi$).  We next show that the local sparse recovery holds under some RIP conditions when $\phi(t) = |t|^q$ for some $q \in(0, 1]$. Before proceeding, we state two useful lemma as follows. The proof of the first lemma is similar to that of \cite[Theorem 3.1]{Cai2013}. The second lemma can be proven similarly as \cite[Theorem 1]{Cai2014} and \cite[Theorems 1 and 2]{Song2014}. Due to the paper length limitation, we omit the proof of them.   

\begin{lemma} \label{RIP-k}
Let $k \in \{1,2,\ldots, n-1\}$ and $q\in (0,1]$ be given. Suppose  
$\delta^A_k  < 1/3$. Then there holds
\[
\|h_{-\max(k)}\|_q^q > \|h_{\max(k)}\|_q^q \quad\quad \forall h \in \mathcal{N}(\A) \backslash \{ 0 \}.
\]
\end{lemma}

\begin{lemma} \label{RIP-tk}
Let $k \in \{1,2,\ldots, n-1\}$ and $q\in (0,1]$ be given. Suppose 
$\delta^A_{\gamma k}  < \frac{1}{\sqrt{(\gamma-1)^{1-2/q}+1}}$ for some $\gamma>1$. Then there holds
\[
\|h_{-\max(k)}\|_q^q > \|h_{\max(k)}\|_q^q \quad\quad \forall h \in \mathcal{N}(\A) \backslash \{ 0 \}.
\]
\end{lemma}

We are now ready to establish a local sparse recovery result for \eqref{p2} with $\phi(t) = |t|^q$ for $q \in(0, 1]$  under some suitable RIP conditions.

\begin{theorem} \label{th5}
Le $q\in (0,1]$ be given and $K$ be defined in 
\eqref{l0-opt}.  Assume that $\delta^A_{K-\left\lfloor r/2 \right\rfloor} < \frac{1}{3}$ or $\delta^A_{\gamma(K-\left\lfloor r/2\right\rfloor)} < \frac{1}{\sqrt{(\gamma-1)^{1-2/q}+1}}$ for some $\gamma>1$. Suppose that $x^*$ is a sparsest solution to the system $Ax=b$.  Then $x^*$ is a strictly local minimizer of problem \eqref{p2} 
with $0\le r\le K$ and $\phi(t)=|t|^q$.
\end{theorem}

\begin{proof}
Since $\delta^A_{K-\left\lfloor r/2 \right\rfloor} < 1/3$ or $\delta^A_{\gamma(K-\left\lfloor r/2\right\rfloor)} < \frac{1}{\sqrt{(\gamma-1)^{1-2/q}+1}}$ for some $\gamma>1$, it follows from 
Lemmas \ref{RIP-k} and \ref{RIP-tk} with $k=K-\left\lfloor r/2\right\rfloor$ that 
\beq \label{ineq7}
\|h_{-\max(K-\left\lfloor r/2 \right\rfloor)}\|_q^q > \|h_{\max(K-\left\lfloor r/2\right\rfloor)}\|_q^q\quad \quad \forall h \in \mathcal{N}(\A) \backslash \{ 0 \}.
\eeq
In addition, we observe that
\[
\ba{lcl}
\left\|h_{-\max(K-\left\lfloor r/2 \right\rfloor)}\right\|_q^q &=& \|h_{-\max(K)}\|_q^q + \sum\limits_{i=K-\left\lfloor r/2 \right\rfloor +1}^K |h|_{[i]}^q,\\ [15pt]
 \left\|h_{\max(K-\left\lfloor r/2 \right\rfloor)}\right\|_q^q &=& \|h_{\max(K-r)}\|_q^q + \sum\limits_{i=K-r+1}^{K-\left\lfloor r/2 \right\rfloor} |h|_{[i]}^q.
\ea
\]
By these relations, \eqref{ineq7} and the fact $\sum\limits_{i=K-r+1}^{K-\left\lfloor r/2 \right\rfloor} |h|_{[i]}^q \geq \sum\limits_{i=K-\left\lfloor r/2 \right\rfloor +1}^K |h|_{[i]}^q$, we have   that for all $h \in \mathcal{N}(\A) \backslash \{ 0 \}$,
\beqa
\|h_{-\max(K)}\|_q^q - \|h_{\max(K-r)}\|_q^q  &=& \|h_{-\max(K-\left\lfloor r/2 \right\rfloor)}\|_q^q-\|h_{\max(K-\left\lfloor r/2 \right\rfloor)}\|_q^q \nn \\ 
&& +\sum\limits_{i=K-r+1}^{K-\left\lfloor r/2 \right\rfloor} |h|_{[i]}^q 
-\sum\limits_{i=K-\left\lfloor r/2 \right\rfloor +1}^K |h|_{[i]}^q \ > \ 0. 
\label{hK-ineq}
\eeqa
Let $\cJ$ be defined in \eqref{cJ}. Then for any $J\in\cJ$, one has 
$J \subseteq \cI_{r+1}^=(x^*)$ and $|J|=\|x^*\|_0-r-|\cI_{r+1}^{<+}(x^*)|$. Using these and $\|x^*\|_0=K$, we obtain that for all $h\in\Re^n$, 
\begin{align*}
\sum\limits_{i \in J \cup \cI_{r+1}^{<+}(x^*)} |h_i|^q &  \ \leq  \
\|h_{\max(\|x^*\|_0-r)}\|^q_q \ =\ \|h_{\max(K-r)}\|^q_q,\\
\sum\limits_{i \in \cI_0(x^*)} |h_i|^q & \ \geq \ \|h_{-\max(\|x^*\|_0)}\|^q_q \ = \ \|h_{-\max(K)}\|^q_q .
\end{align*}
It then follows from these inequalities and \eqref{hK-ineq} that for all $h \in \mathcal{N}(\A) \backslash \{ 0 \}$, 
\beq \nn
\sum\limits_{i\in \cI_0(x^*)} |h_i|^q -\max\limits_{J\in\cJ} \sum\limits_{i\in J \cup \cI_{r+1}^{<+}(x^*)} |h_i|^q \ \geq \ \|h_{-\max(K)}\|^q_q - \|h_{\max(K-r)}\|^q_q \ > \ 0.
\eeq
This implies that {\rm{LNSP}}($r$,$\phi$) holds at $x^*$ for $\phi(t)=|t|^q$. In addition, it is not hard to observe that such $\phi$ satisfies the assumptions stated in Theorem \ref{th4}. The conclusion of this theorem then follows from Corollary \ref{K-loc-recovery}.
\end{proof}

\gap

{\bf Remark:}  Cai and Zhang \cite{Cai2013,Cai2014} recently established sparse recovery by model \eqref{p2} with $r=0$ and $\phi(\cdot)=|\cdot|$ under the condition $\delta^A_k < 1/3$ or $\delta^A_{\gamma k}  < \frac{1}{\sqrt{(\gamma-1)^{-1}+1}}$. In addition, Song and Xia \cite{Song2014} established sparse recovery by model \eqref{p2} with $r=0$ and $\phi(\cdot)=|\cdot|^q$ for $q\in(0,1]$ under the condition  $\delta^A_{\gamma k}   < \frac{1}{\sqrt{(\gamma-1)^{1-2/q}+1}}$ for some $\gamma>1$. As seen from Theorem \ref{th5}, our RIC bounds 
for local sparse recovery are weaker than their bounds when $0<r \le K$. Clearly, there exist some problems satisfying our RIP bounds but violating their bounds. 


\subsection{Global sparse recovery by \eqref{p2}} 

In this subsection we study some sufficient conditions for global recovery of 
a sparsest solution $x^*$ to the system $Ax=b$ by model \eqref{p2}. In particular, we show that if a global null space property holds, $x^*$ is a  
global minimizer of \eqref{p2}. We also establish a global recovery result 
for \eqref{p2} under some RIP conditions and the assumption $\phi(\cdot)=|\cdot|^q$ for $q\in(0,1]$. To proceed, we first introduce the definition of 
global null space property.

\begin{definition}[global null space property]
The {\it global null space property} {\rm{GNSP}}($r$,$\phi$) holds 
for the system $Ax=b$ at a solution $x^*$ with $\|x^*\|_0 \ge r$ if 
\beq \label{gnsp}
\sum\limits_{i \in J_0} \phi (|h_i|)-\sum\limits_{i \in J_1} \phi(|h_i|) > 0 \quad\quad \forall h \in \mathcal{N}(\A) \backslash \{ 0 \} 
\eeq
holds for every $(J_0,J_1)\in\widetilde \cJ$, where  
\beq \label{tcJ}
\widetilde \cJ = \left\{(J_0, J_1): J_0 \subseteq \cI_0(x^*), \  J_1 \subseteq \cI^\c_0(x^*), \ 
|J_0| + |J_1| = n-r
\right\}.
\eeq
\end{definition}

\vgap

We next show that for a class of $\phi$, a solution $x^*$ of $\A x = \b$ is a unique global minimizer of \eqref{p2} if the global null space property holds at $x^*$.

\begin{theorem} \label{th7}
Suppose that Assumption \ref{assump-phi} and  \eqref {phi1} hold for $\phi$. Let $x^* \in \Re^n$ be such that $\A x^* = \b$. Assume that  {\rm{GNSP}}($r$,$\phi$) holds at $x^*$ for some $0\le r \le \|x^*\|_0$. Then $x^*$ is a unique global minimizer of problem \eqref{p2} with such $\phi$ and $r$.
\end{theorem}

\begin{proof}
One can observe that for every $h\in\Re^n$, there exist some $J_0 \subseteq \cI_0(x^*)$ and $J_1 \subseteq \cI_0^\c(x^*)$ (dependent on $h$) with $|J_0| + |J_1| = n-r$ such that 
\beq \nn
\sum\limits_{i=r+1}^n \phi(|x^*+h|_{[i]}) = \sum\limits_{i \in J_1} \phi(|x^*_i + h_i|) + \sum\limits_{i \in J_0} \phi(|h_i|).
\eeq
We also notice that \eqref{phi-tineq} holds here due to \eqref {phi1} and the monotonicity of $\phi$. 
It then follows from the above equality and \eqref{phi-tineq} that
\beq \label{ineq1}
\sum\limits_{i=r+1}^n \phi(|x^*+h|_{[i]}) \geq \sum\limits_{i\in J_1}\phi(|x^*_i|) - \sum\limits_{i\in J_1}\phi(|h_i|) + \sum\limits_{i\in J_0} \phi(|h|_{i}).
\eeq
Observe that $\phi(|x^*|_{[i]})=0$ for all $i>\|x^*\|_0$ due to $\phi(0)=0$. 
Since $J_0 \subseteq \cI_0(x^*)$, one has $|J_0| \leq n-\|x^*\|_0$, which together with $|J_0| + |J_1| = n-r$ implies $|J_1| \geq \|x^*\|_0-r$. Using this, $J_1 \subseteq \cI_0^\c(x^*)$ and $\phi(|x^*|_{[i]})=0$ for all $i>\|x^*\|_0$, one can observe that 
\beq \label{phixs-lbd}
\sum\limits_{i\in J_1} \phi(|x^*_i|) \geq \sum\limits_{i=r+1}^{\|x^*\|_0} \phi(|x^*|_{[i]}) = \sum\limits_{i=r+1}^n \phi(|x^*|_{[i]}).
\eeq
In addition,  we see that $(J_0,J_1)\in\widetilde \cJ$, where $\widetilde\cJ$ is given in \eqref{tcJ}. It thus follows from 
the global null space property  that \eqref{gnsp} holds for the above $(J_0,J_1)$.  In view of \eqref{gnsp}, \eqref{ineq1} and \eqref{phixs-lbd},  we obtain that  
\[
\sum\limits_{i=r+1}^n \phi(|x^*+h|_{[i]}) > \sum\limits_{i=r+1}^n \phi(|x^*|_{[i]}) \quad\quad \forall h \in \mathcal{N}(\A) \backslash \{ 0 \},
\]
and hence $x^*$ is a unique global minimizer of problem \eqref{p2}.
\end{proof}

\gap

As an immediate consequence,  we obtain that under some suitable 
assumption, a sparsest solution $x^*$ to the system $Ax=b$ is a unique global minimizer of \eqref{p2}. 

\begin{corollary}[global sparse recovery] \label{K-g-recovery}
Suppose that Assumption \ref{assump-phi} and  \eqref {phi1} hold for $\phi$, and $x^*$ is a sparsest solution to the system $Ax=b$. Assume that the global null space property {\rm{GNSP}}($r$,$\phi$) holds at $x^*$ for some $0\le r\le K$, where $K$ is defined in \eqref{l0-opt}.  Then $x^*$ is a unique global minimizer of problem \eqref{p2} with such $\phi$ and $r$.
\end{corollary}

The above global sparse recovery is established based on the global null space property {\rm{GNSP}}($r$,$\phi$).  We next show that the global  sparse recovery holds under some RIP conditions when $\phi(t) = |t|^q$ for some $q \in(0, 1]$.

\begin{theorem} \label{th8}
Le $q\in (0,1]$ be given and $K$ be defined in 
\eqref{l0-opt}.  Assume that $\delta^A_{K+\left\lfloor r/2 \right\rfloor} < 1/3$ or $\delta^A_{\gamma(K+\left\lfloor r/2\right\rfloor)} < \frac{1}{\sqrt{(\gamma-1)^{1-2/q}+1}}$ for some $\gamma>1$. Suppose that $x^*$ is a sparsest solution to the system $Ax=b$.  Then $x^*$ is a unique global minimizer of problem \eqref{p2} 
with $0 \le r \le K$ and $\phi(t)=|t|^q$.
\end{theorem}
  
\begin{proof}
Since $\delta^A_{K+\left\lfloor r/2 \right\rfloor} < 1/3$ or $\delta^A_{\gamma(K+\left\lfloor r/2\right\rfloor)} < \frac{1}{\sqrt{(\gamma-1)^{1-2/q}+1}}$ for some $\gamma>1$, it follows from 
Lemmas \ref{RIP-k} and \ref{RIP-tk} with $k=K+\left\lfloor r/2\right\rfloor$ that 
\beq \label{rip-ineq}
\|h_{-\max(K+\left\lfloor r/2 \right\rfloor)}\|_q^q > \|h_{\max(K+\left\lfloor r/2\right\rfloor)}\|_q^q\quad \quad \forall h \in \mathcal{N}(\A) \backslash \{ 0 \}.
\eeq
Observe that
\begin{align*}
& \|h_{-\max(K+\left\lceil r/2 \right\rceil)}\|_q^q = \|h_{-\max(K+r)}\|_q^q + \sum\limits_{i=K+\left\lceil r/2 \right\rceil +1}^{K+r} |h|_{[i]}^q,\\ 
& \|h_{\max(K+\left\lceil r/2\right\rceil)}\|_q^q = \|h_{\max(K)}\|_q^q + \sum\limits_{i=K+1}^{K+\left\lceil r/2\right\rceil } |h|_{[i]}^q.
\end{align*}
Using these relations, \eqref{rip-ineq} and the fact  $\sum\limits_{i=K+1}^{K+\left\lceil r/2 \right\rceil } |h|_{[i]}^q \geq \sum\limits_{i=K+\left\lceil r/2 \right\rceil +1}^{K+r} |h|_{[i]}^q$, we see  that for all $h \in \mathcal{N}(\A) \backslash \{ 0 \}$, 
\beqa
\|h_{-\max(K+r)}\|_q^q - \|h_{\max(K)}\|_q^q &=& \|h_{-\max(K+\left\lceil r/2 \right\rceil)}\|_q^q - \|h_{\max(K+\left\lceil r/2 \right\rceil)}\|_q^q \nn \\
 && + \sum\limits_{i=K+1}^{K+\left\lceil r/2 \right\rceil } |h|_{[i]}^q - \sum\limits_{i=K+\left\lceil r/2 \right\rceil +1}^{K+r} |h|_{[i]}^q 
 \ > \ 0. \label{hK-ineq1}
\eeqa
Let $\widetilde\cJ$ be defined in \eqref{tcJ}. For any $(J_0,J_1)\in\widetilde\cJ$, one has $J_0 \subseteq \cI_0(x^*)$ and $J_1 \subseteq \cI_0^\c(x^*)$  with $|J_0| + |J_1| = n-r$. Since $J_1 \subseteq \cI^\c_0(x^*)$, one has $|J_1| \leq \|x^*\|_0$, which together with $|J_0| + |J_1| = n-r$ implies $|J_0| \geq n-\|x^*\|_0-r$. Using these and $\|x^*\|_0=K$, we obtain that for all $h\in\Re^n$, 
\begin{align*}
\sum\limits_{i \in J_0} |h_i|^q & \ \ge \
\|h_{-\max(\|x^*\|_0+r)}\|^q_q \ =\ \|h_{-\max(K+r)}\|^q_q,\\
\sum\limits_{i \in J_1} |h_i|^q & \ \leq \ \|h_{\max(\|x^*\|_0)}\|^q_q \ = \ \|h_{\max(K)}\|^q_q .
\end{align*}
It then follows from these relations and \eqref{hK-ineq1} that for all $h \in \mathcal{N}(\A) \backslash \{ 0 \}$ and $(J_0,J_1)\in\widetilde\cJ$, 
\beq \nn
\sum\limits_{i\in J_0} |h_i|^q -\sum\limits_{i\in J_1} |h_i|^q \ \geq \ \|h_{-\max(K+r)}\|^q_q - \|h_{\max(K)}\|^q_q \ > \ 0.
\eeq
This implies that {\rm{GNSP}}($r$,$\phi$) holds at $x^*$ for $\phi(t)=|t|^q$. In addition, such $\phi$ satisfies the assumptions stated in Theorem \ref{th7}. The conclusion of this theorem then follows from Corollary \ref{K-g-recovery}.
\end{proof}

\subsection{Stable recovery by \eqref{pL1}}

In this subsection we study the stable recoverability of model \eqref{pL1} for a sparsest solution $\tx^*$ to the noisy system $Ax=b+\xi$ with $\|\xi\| \le \sigma$. In particular, we derive an error bound between $\tx^*$ and a global minimizer of \eqref{pL1} with $\phi(\cdot)=|\cdot|$. To proceed, we establish a technical lemma as follows. 

\begin{lemma} \label{lem:errbd}
Suppose that Assumption \ref{assump-phi} and \eqref {phi1} hold for $\phi$. Assume that $x^*$ is an optimal solution of \eqref{pL1}. Let $x \in \Re^n$ be a feasible solution of problem \eqref{pL1} and $h = x-x^*$. Then there holds:
\beq
\sum^n_{i=k+\left\lceil r/2 \right\rceil+1} \phi(|h|_{[i]}) \ \le \  \sum^{k+\left\lceil r/2 \right\rceil}_{i=1} \phi(|h|_{[i]}) +  2 \sum^n_{i=k+1} \phi(|x|_{[i]}) \label{phi-err}
\eeq
for all $k\ge 1$, $r\ge 0$ and $k+r \le n-1$.
\end{lemma}

\begin{proof}
 Since $x$ and $x^*$ are a feasible solution and an optimal solution of \eqref{pL1}, respectively, one has 
\beq \label{phi-opt}
\sum\limits_{i=r+1}^n \phi(|x|_{[i]}) \geq \sum\limits_{i=r+1}^n \phi(|x^*|_{[i]}).
\eeq
Let $J$ be the set of indices corresponding to the $k$ largest entries of $|x|$. Clearly $|J|=k$. 
By the definitions of $J$ and $h$, one can see that 
\beq \label{phixs}
 \sum\limits_{i=r+1}^n \phi(|x^*|_{[i]}) = \sum\limits_{i=r+1}^n \phi(|x-h|_{[i]}) = \sum\limits_{i \in J_1} \phi(|x_i - h_i|) + \sum\limits_{i \in J_2} \phi(|x_i - h_i|),
\eeq
for some $J_1 \subseteq J$, $J_2 \subseteq J^\c$ with $|J_1| + |J_2| = n-r$. Due to \eqref {phi1}, $\phi(0) = 0$ and the monotonicity of $\phi$, one has 
\[
\phi(|x_i - h_i|) \ \ge \ \phi\left(\big||x_i| - |h_i|\big|\right) \ \ge \  \max\left\{\phi(|x_i|) - \phi(|h_i|),  \phi(|h_i|)- \phi(|x_i|)\right\}\quad\quad \forall i.
\]
Using this and \eqref{phixs}, we have
\beqa
\sum\limits_{i=r+1}^n \phi(|x^*|_{[i]}) & \geq & \sum\limits_{i\in J_1}\phi(|x_i|) - \sum\limits_{i\in J_1}\phi(|h_i|) + \sum\limits_{i\in J_2}\phi(|h_i|) - \sum\limits_{i\in J_2}\phi(|x_i|) \nn \\
&  =& \sum\limits_{i\in J_1 \cup J_2}\phi(|x_i|) + \sum\limits_{i\in J_2}\phi(|h_i|) - \sum\limits_{i\in J_1}\phi(|h_i|) - 2 \sum\limits_{i\in J_2}\phi(|x_i|). \label{phi-lbd}
\eeqa
Since $J_1 \subseteq J$, $|J_1| + |J_2| = n-r$ and $|J|=k$, one has 
$|J_1| \le k$ and $|J_2| \ge n-k-r$. These together with the monotonicity of $\phi$ imply that 
\beq \label{phi-h}
\sum\limits_{i\in J_1}\phi(|h_i|)  \leq \sum^k\limits_{i=1} \phi(|h|_{[i]}), \quad\quad \quad
\sum\limits_{i\in J_2}\phi(|h_i|)  \geq  \sum^n\limits_{i=k+r+1} 
\phi(|h|_{[i]}).
\eeq 
In addition, by $|J_1| + |J_2| = n-r$, $J_2 \subseteq J^\c$, the definition of $J$, and the monotonicity of $\phi$, one has 
\[
\sum\limits_{i\in J_1 \cup J_2}\phi(|x_i|) \geq \sum\limits_{i=r+1}^n \phi(|x|_{[i]}), \quad \quad 
\sum\limits_{i\in J_2}\phi(|x_i|) \leq \sum\limits_{i\in J^\c}\phi(|x_i|) = \sum\limits_{i=k+1}^n \phi(|x|_{[i]}).
\]
These relations, \eqref{phi-lbd} and \eqref{phi-h} yield
\[
\sum\limits_{i=r+1}^n \phi(|x^*|_{[i]})  \ge \sum\limits_{i=r+1}^n \phi(|x|_{[i]}) +\sum^n\limits_{i=k+r+1} 
\phi(|h|_{[i]}) - \sum^k\limits_{i=1} \phi(|h|_{[i]}) - 2 \sum\limits_{i=k+1}^n \phi(|x|_{[i]}),
\]
which together with \eqref{phi-opt} implies 
\[
\sum^n_{i=k+r+1} \phi(|h|_{[i]})  \ \le \ \sum^k_{i=1} \phi(|h|_{[i]}) + 2 \sum^n_{i=k+1} \phi(|x|_{[i]}).
\]
Using this relation and the fact $\sum\limits_{i=k+\left\lceil r/2 \right\rceil+1}^{k+r} \phi(|h|_{[i]})  \le  \sum\limits_{i=k+1}^{k+\left\lceil r/2 \right\rceil} \phi(|h|_{[i]})$, we obtain that 
\begin{align*}
& \sum^n_{i=k+\left\lceil r/2 \right\rceil+1} \phi(|h|_{[i]}) - \sum^{k+\left\lceil r/2 \right\rceil}_{i=1} \phi(|h|_{[i]}) \\
& =  \sum^n_{i=k+r+1} \phi(|h|_{[i]}) + \sum\limits_{i=k+\left\lceil r/2 \right\rceil+1}^{k+r} \phi(|h|_{[i]}) - \sum\limits_{i=1}^k \phi(|h|_{[i]}) -\sum\limits_{i=k+1}^{k+\left\lceil r/2 \right\rceil} \phi(|h|_{[i]}) \\
& \le \sum^n_{i=k+r+1} \phi(|h|_{[i]}) - \sum\limits_{i=1}^k \phi(|h|_{[i]}) \ \le \  2 \sum^n_{i=k+1} \phi(|x|_{[i]}). 
\end{align*}
This implies that \eqref{phi-err} holds as desired.
\end{proof}

\gap

We are now ready to establish error bounds on the sparse recovery by \eqref{pL1} with $\phi(\cdot) = |\cdot|$ 
under some RIP conditions. It shall be mentioned that these results can be generalized to the case where $\phi(t)=|t|^q$ with $q\in (0,1)$ by 
using \eqref{phi-err} and the similar techniques used in the proof of  \cite[Theorem 3.3]{Cai2013} and \cite[Theorem 2]{Song2014}. 
 
\begin{theorem} \label{mth1}
Let $k\ge 2$ and $r\ge 0$ be given such that $k+r \le n-1$. Suppose that $x^*$ is an optimal solution of problem \eqref{pL1} with such $r$, $\phi(t)=|t|$ and some $\sigma \ge 0$, and $\tx^* \in \Re^n$ satisfies $\|A \tx^*-b\|\leq \sigma$. The following statements hold:
\bi
\item[(i)]
 If $\delta=\delta^A_{k+\left\lceil r/2 \right\rceil} < \frac{1}{3}$, then 
\[ 
\|\tx^*-x^*\|\ \leq \ \frac{2\sqrt{2(1+\delta)}\sigma}{1-3\delta} + \frac{2\sqrt{2}(2\delta+\sqrt{(1-3\delta)\delta})}{1-3\delta} \frac{\|\tx^*_{-\max(k)}\|_1}{\sqrt{k+\left\lceil r/2 \right\rceil}};
\]
\item[(ii)]
If $\delta=\delta^A_{\gamma (k+\left\lceil r/2 \right\rceil)} < \sqrt{(\gamma-1)/\gamma}$ for some $\gamma>1$, then 
\[
\|\tx^*-x^*\|\ \leq \ \frac{2\sqrt{2(1+\delta)}\sigma}{1-\sqrt{\gamma/(\gamma-1)}\delta} + 
\left(\frac{\sqrt2\delta+\sqrt{\gamma(\sqrt{(\gamma-1)/\gamma}-\delta)\delta}}{\gamma(\sqrt{(\gamma-1)/\gamma}-\delta)}+1\right) \frac{2\|\tx^*_{-\max(k)}\|_1}{\sqrt{k+\left\lceil r/2 \right\rceil}}.
\] 
\ei
\end{theorem}

\begin{proof}
Let $h = \tx^*-x^*$. In view of \eqref{phi-err} with $\phi(t)=|t|$ 
and $x=\tx^*$, we obtain that 
\[
\|h_{-\max(k+\left\lceil r/2 \right\rceil)}\|_1 \le \|h_{\max(k+\left\lceil r/2 \right\rceil)}\|_1 + 2 \|\tx^*_{-\max(k)}\|_1.
\]
The conclusion of this theorem then follows from this relation and \cite[Theorem 3.3]{Cai2013} and \cite[Theorem 2.1]{Cai2014}.
\end{proof}

\gap

{\bf Remark:} If $\tx^*$ is a $k$-sparse vector satisfying $\|Ax^*-b\| \le \sigma$, the second term in the above error bounds vanishes due to $\tx^*_{-\max(k)}=0$.

\section{Numerical algorithms}
\label{alg}

In this section we study numerical algorithms for solving models  \eqref{pL1} and \eqref{p2}. In particular, we propose a feasible augmented Lagrangian (FAL) method for solving them, which solves a sequence of partially regularized unconstrained optimization problems in the form of 
\beq \label{pruo}
\min_{x \in \Re^n} \left \{F(x) := f(x) + \lambda \sum\limits_{i=r+1}^n \phi(|x|_{[i]}) \right \}.
\eeq

\subsection{Partially regularized unconstrained optimization}
\label{npg-method}

In this subsection we propose a nonmonotone proximal gradient (NPG) method for solving problem \eqref{pruo}, which arises as a subproblem in our subsequent FAL method, and establish its convergence. We also show that  its associated proximal subproblems generally can be solved efficiently. Before proceeding, we make the following assumption on problem \eqref{pruo} throughout this subsection.

\begin{assumption} \label{assump-pruo}
\bi
\item[(i)] $f$ is continuously differentiable in $\cU(x^0;\Delta)$ for some
$x^0\in\Re^n$ and $\Delta>0$, and moreover, there exists some $L_f>0$ such that
\[
\|\nabla f(x)-\nabla f(y)\|\le L_f \|x-y\|, \quad\quad \forall x,y \in \cU(x^0;\Delta),
\]
where
\beqas
 \cU(x^0,\Delta)  &:= & \left\{x: \|x-z\|\le \Delta \ \mbox{for some} \ z \in \cS(x^0) \right\}, \\ [6pt]
 \cS(x^0)  &:=&  \left\{x\in\Re^n: \ F(x) \le F(x^0)\right\}.
\eeqas
\item[(ii)] $F$ is bounded below and uniformly continuous in $\cS(x^0)$.
\item[(iii)] The quantities $A$ and $B$ defined below are finite:
\[
A := \sup\limits_{x \in \cS(x^0)} \|\nabla f(x)\|,  \quad  B := \sup\limits_{x \in \cS(x^0)} \sum\limits_{i=r+1}^n \phi(|x|_{[i]}).
\]
\ei
\end{assumption}

Following a general framework proposed by Wright et al.\ \cite{wn2009}, we now present an NPG method for \eqref{pruo} as follows. The detailed explanation of this method can be found in \cite{wn2009}.

\gap

\noindent
{\bf Algorithm 1: A nonmonotone proximal gradient (NPG) method for \eqref{pruo}} \\ [5pt]
Choose  $0< L_{\min} < L_{\max}$, $\tau>1$, $c>0$ and integer $N \ge 0$ arbitrarily. 
Set $k = 0$.
\begin{itemize}
\item[1)] Choose $L^0_k \in [L_{\min}, L_{\max}]$ arbitrarily. Set $L_k = L^0_k$. 
\bi
\item[1a)] Solve the proximal subproblem
\beq \label{subprob}
    x^{k+1} \in \underset{x \in \Re^n}{\Argmin}\left \{f(x^k)+\nabla f(x^k)^T(x-x^k)+\frac{L_k}{2}\|x-x^k\|^2+\lambda \sum^n_{i=r+1}\phi(|x|_{[i]}) \right \}.
\eeq
\item[1b)] If
\begin{eqnarray}
F(x^{k+1})  \leq \max\limits_{[k-N]^+\leq i \leq k} F(x^i)  - \frac{c}{2}\|x^{k+1}-x^k\|_2^2 \label{descent}
\end{eqnarray}
is satisfied, then go to Step 2).
\item[1c)] Set $L_k\leftarrow\tau L_k$ and go to (\textbf{1a}). 
\ei
\item[2)] Set $\bL_k = L_k$, $k \leftarrow k+1$ and go to Step 1).
\end{itemize}
\noindent
{\bf end}

%
%
%
%
%
%
\vgap

{\bf Remark:} 
\begin{itemize}
\item[(i)]
When $N=0$, $\{F(x^k)\}$ is decreasing. Otherwise, this sequence  may increase at some  
iterations. Thus the above method is generally a nonmonotone method when $N >0$.
\item[(ii)] A practical choice of $L^0_k$ is by the following formula proposed by Barzilai and Borwein \cite{BB}
(see also \cite{E.G.Birgin}):
\beq \label{Lk0}
L^0_k  = \max \left\{ L_{\min } ,\min \left\{L_{\max } ,\frac{\langle s^k, y^k \rangle}
{\|s^k\|^2}\right\} \right\},
\eeq
where $s^k  = x^k  - x^{k - 1}$, $y^k=\nabla f(x^k)-\nabla f(x^{k - 1})$.
\item[(iii)] The main computational burden of Algorithm 1 lies in solving the proximal subproblem \eqref{subprob}. As shown at the end of this subsection,  \eqref{subprob} can be solved as $n-r$ number of one-dimensional problems in the form of 
\beq \label{vdef}
\nu (t) = \min_{u \in \Re} \left \{ \frac{1}{2} (u-t)^2 + \phi(|u|) \right \}
\eeq
for some $t \in \Re$. Clearly the latter problem can be efficiently solved for some commonly used $\phi$ such as $\ell_0$, $\ell_1$,  $\ell_q$,   Capped-$\ell_1$, Log,  MCP and SCAD.
\end{itemize}

\vgap

The following theorem states the number of inner iterations is uniformly 
bounded at each outer iteration of Algorithm 1 and thus this algorithm is 
well defined.  Its proof follows from that of \cite[Proposition A.1]{clp2014} with 
$C= \sup\limits_{x \in \cS(x^0)} \sum\limits_{i=r+1}^n \phi(|x|_{[i]}) \ge 0$. 

\begin{theorem} \label{le32}
Let $\{x^k\}$ and $\bar L_k$ be generated in Algorithm 1, and let
\beq \nn
\bar L := \max\{L_{\max},\tau\underline L, \tau(L_f +c)\}, \ \quad\quad \ \underline L :=  \frac{2(A\Delta+B)}{\Delta^2},
\eeq
where $A$, $B$, $L_f$ and $\Delta$ are given in Assumption \ref{assump-pruo}. 
Under Assumption \ref{assump-pruo}, the following statements hold:
\begin{itemize}
\item[(i)] For each $k \ge 0$, the inner termination criterion \eqref{descent} is satisfied after
at most
\[
\left\lfloor \frac{\log \bar L-\log L_{\min}}{\log \gamma} +1\right\rfloor
\]
inner iterations;
\item[(ii)]  $F(x^k) \le F(x^0)$ and $\bar L_k \le \bar L$ for all $k \ge 0$.
\end{itemize}
\end{theorem}

Before studying the global convergence of the NPG method for solving problem \eqref{pruo}, we introduce two definitions as follows, which can be found in \cite{Rock98}. In addition, for convenience of presentation, let $\Phi(x) := \sum^n_{i=r+1}\phi(|x|_{[i]})$.

\begin{definition}[first-order stationary point]
$x^*\in\Re^n$ is a first-order stationary point of \eqref{pruo} if
\begin{equation}\nn
  0\in \nabla f(x^*)+\lambda\ \partial\Phi(x^*).
\end{equation}
\end{definition}

\begin{definition}[first-order $\eps$-stationary point] \label{eps-stat}
Given any $\eps \ge 0$, $x^*\in\Re^n$ is a first-order $\eps$-stationary point of \eqref{pruo} if
\[
\dist(0,\nabla f(x^*)+\lambda\ \partial\Phi(x^*)) \le \eps.
\]
\end{definition}

We are now ready to establish the global convergence of the NPG method. In particular, we show that any accumulation point of $\{x^k\}$ is a first-order stationary point of problem \eqref{pruo}.

\begin{theorem} \label{npg-converge}
Let the sequence $\{x^k\}$ be generated by Algorithm 1. There holds:
\bi
\item[(i)] $\|x^{k+1}-x^k\| \to 0$ as $k \to \infty$;
\item[(ii)] Any accumulation point of $\{x^k\}$ is a first-order stationary point of \eqref{pruo};
\item[(iii)] For any $\eps>0$, $x^k$ is a first-order $\eps$-stationary point of problem \eqref{pruo} when $k$ is sufficiently large. 
\ei
\end{theorem}

\begin{proof}
(i) The proof of statement (i) is similar to that of \cite[Lemma 4]{wn2009}.

(ii)  It follows from Theorem \ref{le32} that $\{\bL_k\}$ is bounded. By the first-order optimality condition of \eqref{subprob}, we have
\begin{equation}\label{beforelim}
  0\in \nabla f(x^k)+ \bar{L}_k(x^{k+1}-x^k)+\lambda~\partial\Phi(x^{k+1}),
\end{equation}
where $\Phi(x) = \sum^n_{i=r+1}\phi(|x|_{[i]})$. Let $x^*$ be an accumulation point of 
$\{x^k\}$. Then there exists a subsequence $\mathcal{K}$ such that 
$\{x^k\}_{\mathcal{K}}\to x^*$, which together with $\|x^{k+1}-x^k\|\to 0$ implies that 
$\{x^{k+1}\}_{\mathcal{K}}\rightarrow x^*$. Using this, the boundedness of $\{\bar{L}_k\}$, 
the continuity of $\nabla f$ and the outer semi-continuity of $\partial\Phi(\cdot)$, and 
taking limits on both sides of (\ref{beforelim}) as $k\in\mathcal{K}\rightarrow\infty$, we obtain 
\begin{equation*}
  0\in \nabla f(x^*)+\lambda~\partial\Phi(x^*).
\end{equation*}
Hence, $x^*$ is a first-order stationary point of \eqref{pruo}.

(iii) It follows from \eqref{beforelim} with $k$ replaced by $k-1$ that 
\[
\nabla f(x^k)-\nabla f(x^{k-1}) + \bL_{k-1}(x^{k-1}-x^k) \in \nabla f(x^k)+\lambda~\partial\Phi(x^{k}),
\]
which yields 
\beq \label{eps-subdiff1}
\dist(0, \nabla f(x^k)+\lambda~\partial\Phi(x^{k})) \le \|\nabla f(x^{k-1})-\nabla f(x^k)+\bL_k(x^k-x^{k-1})\| .
\eeq 
In addition, one can observe from Theorem \ref{le32} that $\{x^k\} \subseteq \cS(x^0) = \{x: F(x) \le F(x^0)\}$. By Assumption \ref{assump-pruo}, 
we know that $\nabla f$ is uniformly continuous in $\cS(x^0)$. It follows from these and 
Theorem \ref{npg-converge} (i) that 
\beq \label{dist-x}
\|\nabla f(x^{k-1})-\nabla f(x^k)+\bL_k(x^k-x^{k-1})\| \le \eps
\eeq
holds when $k$ is sufficiently large. By \eqref{eps-subdiff1} and \eqref{dist-x}, we see that  for sufficiently large $k$, there holds
\beq \label{eps-subdiff2}
\dist(0, \nabla f(x^k)+\lambda~\partial\Phi(x^{k})) \le \eps,
\eeq 
and hence $x^k$ is a first-order $\eps$-stationary point of problem \eqref{pruo}.
\end{proof}

\vgap

{\bf Remark:} 
Recall that $\Phi(x) = \sum^n_{i=r+1}\phi(|x|_{[i]})$. Due to the complication of this function, it is generally hard to evaluate $\partial \Phi$ and hence verify \eqref{eps-subdiff2} directly. Nevertheless, one 
can see from the above proof that \eqref{dist-x} implies \eqref{eps-subdiff2}. Therefore, \eqref{dist-x} can be used as a verifiable termination criterion for Algorithm 1 to generate a first-order $\eps$-stationary point of \eqref{pruo}.  

\vgap

In the remainder of this subsection, we discuss how to efficiently solve the proximal subproblem \eqref{subprob} of Algorithm 1. In particular, we show that \eqref{subprob} can be solved as $n-r$ number of one-dimensional problems in the form of \eqref{vdef}. To proceed,  we first study some properties of the function $\nu (t)$ that is defined in \eqref{vdef}. The following lemma shows that $\nu$ is locally Lipschitz everywhere, and it also characterizes the Clarke subdifferential of $\nu$. 
 
\begin{lemma} \label{lem10}
Let $\nu$ be defined in \eqref{vdef}. Suppose that $\phi$ satisfies Assumption \ref{assump-phi}. Then  the following properties hold:
\begin{itemize}
\item [(i)] $\nu$ is well-defined in $\Re$;
\item [(ii)] $\nu$ is locally Lipschitz in $\Re$;
\item [(iii)] $\partial^\c \nu(t) = \co(\{t-u: u \in \M(t)\})$, where $\partial^\c \nu(t)$ denotes the Clarke subdifferential of $\nu$ at $t$, and $\M(t)$ is the set of the optimal solutions of \eqref{vdef} at $t$, that is, 
\beq \label{Mt}
\M (t) = \left\{ u \in \Re : \frac{1}{2} (u-t)^2 + \phi(|u|) = \nu (t) \right\}. 
\eeq
\end{itemize}
\end{lemma}

\begin{proof}
(i) The well-definedness of $\nu$ follows from Theorem \ref{th1}.

(ii) Let $t_1 < t_2$ be arbitrarily chosen. For any $t \in [t_1, t_2]$ and $u^* \in \M(t)$, it follows from \eqref{vdef} and $\phi (0) = 0$ that 
\[
\frac{1}{2} (u^*-t)^2 + \phi(|u^*|) = \nu (t) \leq \frac{1}{2} (0-t)^2 + \phi(0) = \frac{1}{2} t^2,
\]
with together with $\phi(|u^*|) \geq 0$ implies $|u^*-t| \le |t|$. By this and 
$t \in [t_1, t_2]$, one has
\beq \label{ineq5}
-2 \max \{|t_1|,|t_2|\} \leq t-|t| \leq u^* \leq t+|t| \leq 2 \max \{|t_1|,|t_2|\}.
\eeq
Hence, we obtain that for any $t \in [t_1, t_2]$, 
\beq \label{u1}
|u^*| \leq 2 \max \{|t_1|,|t_2|\}\quad \quad \forall u^* \in \M(t).
\eeq
For any $\tilde{t}_1,\tilde{t}_2 \in [t_1, t_2]$, let $u^*_1 \in \M(\tilde{t}_1)$ and $u^*_2 \in \M(\tilde{t}_2)$. In view of \eqref{vdef} and \eqref{u1}, we can see that
\begin{align*}
\nu (\tilde{t}_1) & \leq \frac{1}{2} (u^*_2-\tilde{t}_1)^2 + \phi(|u^*_2|) = \frac{1}{2} (u^*_2-\tilde{t}_2)^2 + \phi(|\tilde{u}^*_2|) + (\tilde{t}_2-\tilde{t}_1)\left(u^*_2- \frac{\tilde{t}_1+\tilde{t}_2}{2}\right)\\
& \leq \nu(\tilde{t}_2) + |\tilde{t}_1-\tilde{t}_2| \left(|u^*_2|+\frac{1}{2}|\tilde{t}_1|+\frac{1}{2}|\tilde{t}_2|\right) \leq \nu(\tilde{t}_2) + 3\max \{ |t_1|,|t_2| \} |\tilde{t}_1-\tilde{t}_2|.
\end{align*}
Analogously, one can show that
\beq \nn
\nu (\tilde{t}_2) \leq \nu (\tilde{t}_1) + 3\max \{|t_1|,|t_2|\} |\tilde{t}_1-\tilde{t}_2|.
\eeq
It then follows that 
\beq \nn
|\nu (\tilde{t}_1) - \nu (\tilde{t}_2)| \leq 3\max \{|t_1|,|t_2|\} |\tilde{t}_1-\tilde{t}_2| \quad \quad \forall \tilde{t}_1,\tilde{t}_2 \in [t_1, t_2], 
\eeq
and hence statement (\romannumeral2) holds.

(iii) Since $\nu$ is locally Lipschitz everywhere, it is differentiable almost everywhere, and moreover, $\partial^\c \nu(t)=\co(S(t))$ for all $t\in\Re$, where  
\[
S(t) = \left\{\xi: \exists t_k  \to t \ \mbox{such that} \ \nu'(t_k) \to \xi \right\} \quad\quad \forall t \in \Re,
\]
and $\nu'(t_k)$ denotes the standard derivative of $\nu$ at $t_k$. 

Given any $t \in \Re$, $u^* \in \M(t)$ and $h \in \Re$, we have 
\begin{align*}
(t-u^*)h & = \lim_{\delta \downarrow 0} \frac{1}{\delta}\left[\frac{1}{2} \|t + \delta h - u^*\|^2 - \frac{1}{2} \|t - u^*\|^2\right]\\
& = \lim_{\delta \downarrow 0} \frac{1}{\delta}\left[\left(\frac{1}{2} \|t + \delta h - u^*\|^2 + \phi(|u^*|)\right) - \left(\frac{1}{2} \|t - u^*\|^2 + \phi(|u^*|)\right)\right]\\
& \geq \liminf_{\delta \downarrow 0} \frac{\nu (t+\delta h) - \nu(t)}{\delta},
\end{align*} 
where the last inequality follows from $u^* \in \M(t)$ and the definition of $\nu$.
Using this inequality and \cite[Corollary 1.10]{Clarke1975}, we obtain that $t-u^* \in \partial^\c \nu (t)$. Due to the arbitrarity of $u^* \in \M(t)$ and the convexity of $\partial^\c \nu (t)$,  one further has
\beq \label{ineq3}
\co(\{t-u: u \in \M(t)\}) \subseteq \partial^\c \nu (t).
\eeq

It remains to show $\partial^\c \nu (t) \subseteq \co(\{t-u: u \in \M(t)\})$. Indeed, let $\xi \in S(t)$ be arbitrarily chosen. Then there exists some sequence $\{ t_k \}$ such that $t_k\to t$ and $\nu' (t_k)\to \xi$. Choose $u_k \in \M(t_k)$ arbitrarily and let $\ut$, $\bt$ be such that $\{t_k\} \subseteq [\ut, \bt]$. From \eqref{ineq5}, we know that $|u_k| \leq 2 \max \{|\ut|,|\bt|\}$ and hence $\{u_k\}$ is bounded.  Considering a convergent subsequence if necessary, assume without loss of generality that $u_k \to u^*$ for some $u^*\in\Re$. It follows from $u_k \in \M(t_k)$ and the definition of $\M(\cdot)$ that
\beq \nn
\frac{1}{2} (u-t_k)^2 + \phi(|u|) \geq \frac{1}{2} (u_k-t_k)^2 + \phi(|u_k|) \quad\quad \forall u \in \Re.
\eeq
Taking $\liminf$ on both sides of this relation as $k \to \infty$ and using the lower semi-continuity of $\phi$, we have 
\[
\frac{1}{2} (u-t)^2 + \phi(|u|) \geq \liminf_{k \to \infty} \left(\frac{1}{2} (u_k-t_k)^2 + \phi(|u_k|)\right) \ge \frac{1}{2} (u^*-t)^2 + \phi(|u^*|)\quad\quad \forall u \in \Re.
\]
This yields $u^* \in \M(t)$. In addition, since $\nu$ is differentiable at $t_k$, one has $\partial^c\nu(t_k)=\{\nu'(t_k)\}$, which together with $u_k \in \M(t_k)$ and \eqref{ineq3} implies $\nu'(t_k)=t_k - u_k$ for all $k$.  
In view of this, $t_k \to t$, $u_k \to u^*$ and $\nu'(t_k) \to \xi$, one has
\[
\xi = \lim_{k \to \infty} \nu'(t_k) = \lim_{k \to \infty} (t_k - u_k) = t - u^*.
\]
It follows from this, $u^* \in \M(t)$ and the arbitrarity of $\xi \in S(t)$ 
that $S(t) \subseteq \{t-u: u \in \M(t)\}$, which along with  the fact $\partial^\c \nu(t)=\co(S(t))$ implies that 
$\partial^\c \nu(t) \subseteq \co(\{t-u: u \in \M(t)\})$. The conclusion of statement (iii) immediately follows from this relation and \eqref{ineq3}.
\end{proof}

\gap

We next show that $\nu$ is increasing with respect to $|t|$.

\begin{lemma} \label{lem11}
Let $\nu$ be defined in \eqref{vdef}. Suppose that $\phi$ satisfies Assumption \ref{assump-phi}.  
 Then $\nu(t_1) \geq \nu(t_2)$ for any $t_1$, $t_2$ such that $|t_1| \geq |t_2|$. 
\end{lemma}

\begin{proof}
Observe that $\nu(-t) = \nu(t)$ for all $t\in\Re$. To prove this lemma, it thus  suffices to show that $\nu(t)$ is increasing on $[0,\infty)$. We first claim that $\M(t) \subseteq (-\infty,t]$ for all $t \ge 0$, where $\M(t)$ is defined in \eqref{Mt}. Suppose for contradiction that there exists some $t \ge 0$ and $u^* \in \M(t)$ such that $u^* > t$. Let $u = (u^* + t)/2$. 
Clearly $0\le t<u<u^*$. By this and the monotonicity of $\phi$ on $[0,\infty)$, on can see that 
\[
\frac{1}{2} (u-t)^2 +\phi(|u|) < \frac{1}{2} (u^*-t)^2 +\phi(|u^*|),
\]
which contradicts $u^* \in \M(t)$. Hence, $\M(t) \subseteq (-\infty,t]$ 
for all $t \ge 0$.  This together with Lemma \ref{lem10} (iii) implies that 
$\partial^\c \nu(t) \in [0,\infty)$ for all $t \ge 0$. Further, by the mean value theorem \cite[Theorem 2.3.7]{Clarke83},  one has
\[
\nu(t_1) - \nu(t_2) \in \partial^\c\nu(\tilde t) (t_1-t_2) \quad\quad \forall t_1,t_2 \ge 0,
\]
where $\tilde t$ is some point between $t_1$ and $t_2$. Hence, $\nu(t)$ is increasing in $[0,\infty)$ and the conclusion of this lemma holds.
\end{proof}

\gap

We are now ready to discuss how to solve efficiently the subproblem \eqref{subprob} of Algorithm 1.  Clearly, \eqref{subprob} is equivalent to 
\[ 
\min_{x \in \Re^n} \left \{ \frac{1}{2} \left \|x-\left(x^k-\frac{1}{L_k} \triangledown f(x^k)\right)\right \|_2^2 + \frac{\lambda}{L_k}\sum\limits_{i=r+1}^n \phi(|x|_{[i]})\right \},
\]
which is a special case of a more general problem 
\beq \label{p5}
\min_{x \in \Re^n} \left \{ \frac{1}{2} \| x - a \|^2_2 + \tilde{\lambda} \sum\limits_{i=r+1}^n \phi(|x|_{[i]}) \right \}, 
\eeq
for some $a \in \Re^n$ and $\tilde{\lambda} > 0$. In what follows, we show that problem \eqref{p5} can be solved as $n-r$ number of one-dimensional problems in the form of \eqref{vdef}.

\begin{theorem}
Suppose that $\phi$ satisfies Assumption \ref{assump-phi}.  Let $\cI^*$ be the index set corresponding to the $n-r$ smallest entries of $|a|$ and $x^*\in \Re^n$ be defined as follows:
$$x^*_i \in \left\{
    \begin{array}{ll}
      \underset{u \in \Re}{\Argmin}\left \{ \frac{1}{2} (u-a_i)^2 + \tilde{\lambda}\phi(|u|) \right \} & \hbox{if $i \in \cI^*$}, \\
      \{a_i\} & \hbox{otherwise}
    \end{array}
  \right.
\quad i = 1,\ldots, n.$$
Then $x^*$ is an optimal solution of problem \eqref{p5}.
\end{theorem}

\begin{proof}
Let $\nu$ be defined in \eqref{vdef}. In view of  the monotonicity of $\phi$ and the definition of $\nu$, one can observe that
\begin{align*}
& \min_{x \in \Re^n} \frac{1}{2} \| x - a \|^2 + \tilde{\lambda} \sum\limits_{i=r+1}^n \phi(|x|_{[i]})  = \min_{x \in \Re^n} \left \{\min_{\substack{s \in \{0,1\}^n \\ \| s \|_0 = n-r}} \frac{1}{2} \| x - a \|^2 + \tilde{\lambda} \sum\limits_{i=1}^n s_i \phi(|x_i|)\right \} \\
& = \min_{\substack{s \in \{0,1\}^n \\ \| s \|_0 = n-r}} \left \{ \min_{x \in \Re^n} \frac{1}{2} \| x - a \|^2 + \tilde{\lambda} \sum\limits_{i=1}^n s_i \phi(|x_i|) \right \} 
 = \min_{\substack{I \subseteq \{1,\ldots,n\} \\ |I| = n-r}} \sum\limits_{i \in I} \nu(a_i) = \sum\limits_{i \in \cI^*} \nu(a_i),
\end{align*}
where the last equality follows from Lemma \ref{lem11} and the definition of $\cI^*$. In addition, by the definition of $x^*$, one can verify that 
\[
\frac{1}{2} \| x^* - a \|^2 + \tilde{\lambda} \sum\limits_{i=r+1}^n \phi(|x^*|_{[i]}) = \sum\limits_{i \in \cI^*} \nu(a_i),
\]
and hence the conclusion holds.
\end{proof}

\subsection{Sparse recovery for a noiseless linear system}
\label{noiseless-sys}

In this subsection we propose a feasible augmented Lagrangian (FAL) method for solving the partially regularized sparse recovery model for a noiseless linear 
system, namely, problem \eqref{p2}. For convenience of presentation, let $\Phi(x) := \sum^n_{i=r+1} \phi(|x|_{[i]})$ throughout this subsection. 



For any given $\varrho>0$ and $\mu\in\Re^m$, the augmented Lagrangian (AL) function for problem  \eqref{p2} is defined as
\begin{eqnarray}\label{auglag}
\cL(x;\mu,\varrho):= w(x;\mu,\varrho)+\Phi(x),
\end{eqnarray}
where 
\beq \label{w}
w(x;\mu,\varrho) := \mu^T(Ax-b)+\frac{\varrho}{2}\|Ax-b\|^2.
\eeq 
Similar to the classical AL method, our FAL method solves a sequence of AL subproblems in the form of
\begin{eqnarray}\label{subp augmented}
 \min\limits_{x} \ \cL(x;\mu,\varrho)
\end{eqnarray}
while updating the Lagrangian multiplier $\mu$ and the penalty parameter $\varrho$. It is easy to observe that problem \eqref{subp augmented}   
 is a special case of  \eqref{pruo} and satisfies Assumption  \ref{assump-pruo}. Therefore, it can be suitably solved by Algorithm 1.

 It is well known that the classical AL method may converge to an infeasible point. To remedy this pathological behavior, we adopt the same strategies proposed by Lu and Zhang   
\cite{int:Lu} so that the values of the AL function along the solution sequence 
generated by the FAL method are bounded from above, and the magnitude of 
penalty parameters outgrows that of Lagrangian multipliers. 

Throughout this section, let $x^{\rm feas}$ be an arbitrary feasible point of problem \eqref{p2}. We are now ready to present the FAL method for solving \eqref{p2}.

\gap

\noindent
{\bf Algorithm 2: A feasible augmented Lagrangian method for \eqref{p2}} \\ [5pt]
Choose $\mu^0\in \Re^m$, $x^0\in\Re^n$,  $\varrho_0>0$, $\gamma\in (1,\infty)$, $\theta>0$, 
$\eta\in (0,1)$, a decreasing positive sequence $\{\epsilon_k\}$ with 
$\lim\limits_{k\to\infty}\epsilon_k=0$, and a constant $\Upsilon\geq \max\{\Phi (x^{\feas}),\cL(x^0,\mu^0,\varrho_0)\}$.  Set $k=0$.
\begin{itemize}
\item[1)] Apply Algorithm 1 to find an approximate stationary point $x^{k+1}$ for problem \eqref{subp augmented} with $\mu=\mu^k$ and $\varrho=\varrho_k$ such that
    \begin{equation}\label{innerstop}
    \dist(0,\nabla  w(x^{k+1};\mu^k,\varrho_k)+\partial \Phi(x^{k+1}))\leq\epsilon_k,\quad\quad \cL(x^{k+1},\mu^k,\varrho_k)\leq \Upsilon.
    \end{equation}
\item[2)] Set $\mu^{k+1}=\mu^k+\varrho_k(Ax^{k+1}-b)$.
\item[3)] If  $\|Ax^{k+1}-b\|\leq\eta\|Ax^{k}-b\|$, then set $\varrho_{k+1}=\varrho_k$. Otherwise, set $$\varrho_{k+1}=\max\{\gamma\varrho_{k},\|\mu^{k+1}\|^{1+\theta}\}.$$
\item[4)]  Set $k\leftarrow k+1$ and go to Step 1).
\end{itemize}
\noindent
{\bf end}

%
%
%
%

\vgap

We now discuss how to find by Algorithm 1 an approximate stationary point 
$x^{k+1}$ of problem \eqref{subp augmented} with $\mu=\mu^k$ and 
$\varrho=\varrho_k$ such that \eqref{innerstop} holds . Given $k \ge 0$, let $x^k_\init$ denote the initial point of Algorithm 1 when 
applied to solve problem \eqref{subp augmented}  with $\mu=\mu^k$ and 
$\varrho=\varrho_k$, which is specified as follows:
\[
x^k_\init = \left\{
\ba{ll}
x^{\rm feas} & \mbox{if} \ \cL(x^{k},\mu^k,\varrho_k) > \Upsilon, \\
x^{k}  &\mbox{otherwise},
\ea \right.
\]
where $x^{k}$ is the approximate stationary point of problem 
\eqref{subp augmented}  with $\mu=\mu^{k-1}$ and 
$\varrho=\varrho_{k-1}$ obtained at the $(k-1)$th outer iteration of Algorithm 2.  Notice that $Ax^{\rm feas}=b$. It follows 
from \eqref{auglag} and the definition of $\Upsilon$  that
\[
\cL(x^{\rm feas};\mu^k,\varrho_k)  =  \Phi(x^{\rm feas})  \ \le \ \Upsilon.
\]
This inequality and the above choice of $x^k_\init$ imply 
$\cL(x^k_\init,\mu^k,\varrho_k) \le \Upsilon$.
In addition, by Theorem \ref{le32}, we know that all subsequent objective 
values of Algorithm 1 are bounded above the initial objective value. It follows that 
\[
\cL(x^{k+1};\mu^k,\varrho_k) \ \le  \ \cL(x^k_\init,\mu^k,\varrho_k)
\ \le \ \Upsilon,
\]
and hence the second relation of \eqref{innerstop} is satisfied at $x^{k+1}$. Finally, as seen 
from Theorem \ref{npg-converge} (iii),  Algorithm 1 is capable of finding an approximate 
solution $x^{k+1}$ satisfying the first relation of \eqref{innerstop}.

\vgap

We next establish a global convergence result for the above FAL method.

\begin{theorem}
Let $\{x^k\}$ be the sequence generated by Algorithm 2. Suppose that $x^*$ is an accumulation point of $\{x^k\}$. Then the following statements hold:
\bi
\item[(i)] $x^*$ is a feasible point of problem \eqref{p2};
\item[(ii)] Suppose additionally that $\partial^\infty \Phi(x^*) \cap \range(A^T)=\{0\}$.
 Then $x^*$ is a 
first-order stationary point of problem \eqref{p2}, that is, $Ax^*=b$ and 
$0 \in A^T \mu^* + \partial \Phi(x^*)$ for some $\mu^*\in\Re^m$.
\ei
\end{theorem}

\begin{proof} 
Suppose that $x^*$ is an accumulation point of $\{x^k\}$. Considering a convergent subsequence if 
necessary, assume without loss of generality that $x^k \to x^*$ as $k\to \infty$. 

(i) We show that $x^*$ is a feasible point of problem \eqref{p2}, that is, 
$Ax^*=b$, by considering two separate cases as follows.

Case (a): $\{\varrho_k\}$ is bounded. This together with Step 3) of Algorithm 2 implies that  
$\|Ax^{k+1}-b\|\leq\eta\|Ax^{k}-b\|$
holds for all sufficiently large $k$. It then follows from this and $x^k\to x^*$ that $Ax^*=b$.

Case (b): $\{\varrho_k\}$ is unbounded. Considering a subsequence if necessary, assume for 
convenience that $\{\varrho_k\}$ strictly increases and $\varrho_k\to\infty$ as $k\to\infty$. This along with 
Step 3) of Algorithm 2 implies that 
$\varrho_{k+1}=\max\{\gamma\varrho_k,\|\mu^{k+1}\|^{1+\theta }\}$ holds for all $k$. It then follows from this and $\rho_k\to\infty$ that
\beq \label{lag-pen}
0 \le \varrho_k^{-1}\|\mu^k\| \le \varrho_k^{-\theta/(1+\theta)}\to 0.
\eeq
In view of \eqref{auglag} and the second relation in \eqref{innerstop}, one has 
\begin{equation*}
(\mu^k)^T(Ax^{k+1}-b)+\frac{\varrho_k}{2}\|Ax^{k+1}-b\|^2+\Phi(x^{k+1}) \leq \Upsilon.
\end{equation*}
It then follows that
\[
\|Ax^{k+1}-b\|^2 \leq 2\varrho_k^{-1}\left(\Upsilon-\Phi(x^{k+1})\right)+2\varrho_k^{-1}\|\mu^k\| \|Ax^{k+1}-b\|.
\]
Taking limits on both sides of this relation as $k\to\infty$, and using \eqref{lag-pen}, $x^k\to x^*$, $\varrho_k\to\infty$ and the lower 
boundedness of $\Phi$,  we obtain $Ax^*=b$ again.

(ii) We show that $x^*$ is a first-order stationary point of problem \eqref{p2}. It follows from \eqref{w}, the first relation in \eqref{innerstop} and Step 2) of Algorithm 2 that there exist some $\xi^k\in \Re^n$ and $v^k \in \partial \Phi(x^{k+1})$ such that $\|\xi^k\| \leq\epsilon_k$ and 
\begin{equation}\label{linear system}
 \xi^k=A^T\mu^{k+1}+v^k.
\end{equation}
Since $\|\xi^k\|\le \epsilon_k$ and $\epsilon_k \to 0$, we have $\|\xi^k\| \to 0$. Claim that $\{v^k\}$ is bounded. Suppose for contradiction that $\{v^k\}$ is unbounded. Considering a subsequence if necessary,  assume for convenience that $\|v^k\| \to \infty$ 
and $v^k/\|v^k\| \to v$ for some $v\in\Re^n$. Clearly $v \neq 0$. By $x^k \to x^*$,  $v^k \in \partial \Phi(x^{k+1})$, $\|v^k\| \to \infty$ and  
$v^k/\|v^k\| \to v$, one can see that  
$v\in\partial^\infty\Phi(x^*)$.  Dividing both sides of \eqref{linear system} by $\|v^k\|$ and taking limits as $k \to \infty$, we obtain $0 \neq v \in \partial^\infty\Phi(x^*) \cap \range(A^T)$, which 
contradicts the assumption $\partial^\infty \Phi(x^*) \cap \range(A^T)=\{0\}$. Thus $\{v^k\}$ is bounded.  Since \eqref{linear system} is a linear system with respect to $\mu^{k+1}$, 
it follows from Hoffman lemma \cite{Hoffman1952} that there exist $\bar{\mu}^{k+1}\in \Re^m$ and a constant 
$\theta>0$ (independent of $k$) such that
\beq
\A^T \bar{\mu}^{k+1}= \xi^k-v^k, \quad\quad \|\bar{\mu}^{k+1}\|\leq \theta \|\xi^k-v^k\|. \label{bound}
\eeq
Using the second relation of \eqref{bound}  and the boundedness of $\{v^k\}$ and $\{\xi^k\}$, one can  
see that $\{\bar{\mu}^{k+1}\}$ is bounded. Considering a subsequence if necessary, assume for convenience 
 that $\bar{\mu}^{k+1}\to \mu^*$ and $v^k \to v^*$ for some $\mu^*\in\Re^m$ and $v^*\in\Re^n $. By $x^k \to x^*$,  $v^k \in \partial \Phi(x^{k+1})$ and the outer semi-continuity of $\partial \Phi$, one has $v^*\in\partial \Phi(x^*)$.  Using $\xi^k \to 0$, $\bar{\mu}^{k+1}\to \mu^*$, $v^k \to v^*\in\partial \Phi(x^*)$ and taking limits on both sides of the first relation of \eqref{bound} as $k\to\infty$, we obtain  $0\in A^T\mu^*+ \partial \Phi(x^*)$. 
This along with $Ax^*=b$ implies that $x^*$ is a first-order stationary 
point of \eqref{p2}.
\end{proof}

\subsection{Sparse recovery for a noisy linear system}
\label{noisy-sys}


In this subsection we propose a FAL method for solving the partially regularized sparse recovery model for a noisy linear 
system, namely, problem \eqref{pL1}. For convenience of presentation, let $\Phi(x) := \sum^n_{i=r+1} \phi(|x|_{[i]})$ throughout this subsection.

For any given $\varrho>0$ and $\mu \ge 0$, the AL function for problem \eqref{pL1} is defined as
\begin{eqnarray}\label{auglagn}
\tL(x;\mu,\varrho):= \tw(x;\mu,\varrho)+\Phi(x),
\end{eqnarray}
where 
\beq \label{tw}
\tw(x;\mu,\varrho) := \frac{1}{2\varrho}\left([\mu+\varrho(\|Ax-b\|^2-\sigma^2)]_+^2-\mu^2\right).
\eeq 
Similar to the classical AL method, our FAL method solves a sequence of AL subproblems in the form of
\begin{eqnarray}\label{aug-subprobn}
 \min\limits_{x} \ \tL(x;\mu,\varrho).
\end{eqnarray}
By a similar argument as in \cite{clp2014}, one can verify that for any given $\mu \ge 0$ and $\rho \ge 0$, 
problem \eqref{aug-subprobn} is a special case of \eqref{pruo} and satisfies Assumption \ref{assump-pruo}. Therefore, 
the NPG method presented in Algorithm 1 can be suitably applied to solve  \eqref{aug-subprobn}.

For a similar reason as mentioned in Subsection \ref{noiseless-sys}, we adopt 
the same strategies proposed in \cite{int:Lu} so that the feasibility of any accumulation point of the solution sequence generated by our 
method is guaranteed. Throughout this section, let $x^{\rm feas}$ be an arbitrary feasible point of problem \eqref{pL1}. We are now ready to present the FAL method for solving \eqref{pL1}.

\gap

\noindent
{\bf Algorithm 3: A feasible augmented Lagrangian method for \eqref{pL1}} \\ [5pt]
Choose $\mu_0 \ge 0$, $x^0\in\Re^n$,  $\varrho_0>0$, $\gamma\in (1,\infty)$, $\theta>0$, 
$\eta\in (0,1)$, a decreasing positive sequence $\{\epsilon_k\}$ with 
$\lim\limits_{k\to\infty}\epsilon_k=0$, and a constant $\Upsilon\geq \max\{\Phi (x^{\feas}),\tL(x^0,\mu_0,\varrho_0)\}$.  Set $k=0$.
\begin{itemize}
\item[1)] Apply Algorithm 1 to find an approximate stationary point $x^{k+1}$ for problem \eqref{aug-subprobn} with $\mu=\mu_k$ and $\varrho=\varrho_k$ such that
    \begin{equation}\label{innerstopn}
    \dist(0,\nabla \tw(x^{k+1};\mu_k,\varrho_k)+\partial \Phi(x^{k+1}))\leq\epsilon_k,\quad\quad \tL(x^{k+1},\mu_k,\varrho_k)\leq \Upsilon.
    \end{equation}
\item[2)] Set $\mu_{k+1}=[\mu_k+\varrho_k(\|Ax^{k+1}-b\|^2-\sigma^2)]_+$.
\item[3)] If  $[\|Ax^{k+1}-b\|-\sigma]_+\leq\eta[\|Ax^k-b\|-\sigma]_+$, then set $\varrho_{k+1}=\varrho_k$. Otherwise, set $$\varrho_{k+1}=\max\{\gamma\varrho_{k},\mu_{k+1}^{1+\theta }\}.$$
\item[4)] Set $k\leftarrow k+1$ and go to Step 1).
\end{itemize}
\noindent
{\bf end}

%
%
%
%
\vgap
By a similar argument as in Subsection \ref{noiseless-sys}, one can observe that $x^{k+1}$ satisfying \eqref{innerstopn} can be found by Algorithm 1. We next establish a global convergence result for the above FAL method.

\begin{theorem}
Let $\{x^k\}$ be a sequence generated by Algorithm 3. Suppose that $x^*$ is  an accumulation point of $\{x^k\}$. Then the following statements hold:
\bi
\item[(i)] $x^*$ is a feasible point of problem \eqref{pL1};
\item[(ii)] Suppose additionally that $\partial^\infty \Phi(x^*) \cap \cC=\{0\}$ and $A^T(Ax^*-b) \neq 0$ 
if $\|Ax^*-b\|=\sigma$, where  
\[
\cC := \{t A^T(b-Ax^*): t \ge 0\}.
\]
Then $x^*$ is a first-order stationary point of problem \eqref{pL1}, that is, there exists $\mu^* \ge 0$ such that 
\beq \label{KKT}
\ba{l}
 0 \in 2\mu^* A^T(Ax^*-b) + \partial \Phi(x^*), \\ [5pt]
 \|Ax^*-b\| \le \sigma, \quad\quad \mu^*(\|Ax^*-b\|-\sigma)=0.  
\ea
\eeq
\ei
\end{theorem}

\begin{proof} 
Suppose that $x^*$ is an accumulation point of $\{x^k\}$. Considering a convergent subsequence if 
necessary, assume for convenience that $x^k\to x^*$ as $k\to \infty$. 

(i) We show that $x^*$ is a feasible point of problem \eqref{p2}, that is, $\|Ax^*-b\| \le \sigma$, by considering two separate cases as follows. 

Case (a): $\{\varrho_k\}$ is bounded. This together with Step 3) of Algorithm 3 implies that  
\[
[\|Ax^{k+1}-b\|-\sigma]_+\leq\eta[\|Ax^k-b\|-\sigma]_+
\]
holds for all sufficiently large $k$. It then follows from this and $x^k\to x^*$  that $[\|Ax^*-b\|-\sigma]_+=0$ and hence 
$\|Ax^*-b\| \le \sigma$. 

Case (b): $\{\varrho_k\}$ is unbounded. Considering a subsequence if necessary, assume for 
convenience that $\{\varrho_k\}$ strictly increases and $\varrho_k \to \infty$ as $k\to\infty$. This along with 
Step 3) of Algorithm 2 implies that 
$\varrho_{k+1}=\max\{\gamma\varrho_k,\|\mu^{k+1}\|^{1+\theta }\}$ holds for all $k$. It then follows from this and $\rho_k \to \infty$ that
\beq \label{lag-penn}
0 \le \varrho_k^{-1}\mu_k\le \varrho_k^{-\theta/(1+\theta)}\to 0.
\eeq
In view of \eqref{auglagn} and the second relation in \eqref{innerstopn}, one has 
\begin{equation*}
\frac{1}{2\varrho_k}\left(\left[\mu_k+\varrho_k\left(\|Ax^{k+1}-b\|^2
-\sigma^2\right)\right]_+^2-\mu^2_k\right) + \Phi(x^{k+1}) \leq \Upsilon.
\end{equation*}
It follows that
\[
\left(\varrho_k^{-1}\mu_k+\|Ax^{k+1}-b\|^2
-\sigma^2\right)_+^2 \leq 2\varrho_k^{-1}[\Upsilon-\Phi(x^{k+1})] + \varrho^{-2}_k\mu^2_k. 
\]
Taking limits on both sides of this relation as $k\to\infty$, and using \eqref{lag-penn}, $x^k\to x^*$, $\varrho_k\to\infty$ and the lower 
boundedness of $\Phi$, we obtain $(\|Ax^*-b\|^2-\sigma^2)_+=0$, 
which clearly yields $\|Ax^*-b\| \le \sigma$.

(ii) We show that $x^*$ is a first-order stationary point of problem \eqref{pL1}. It follows from \eqref{tw}, the first relation in \eqref{innerstopn} and Step 2) of Algorithm 3 that there exist some $\xi^k\in \Re^n$ and $v^k \in \partial \Phi(x^{k+1})$ such that $\|\xi^k\| \leq\epsilon_k$ and  
\begin{equation}\label{linear systemn}
 \xi^k=2\mu_{k+1} A^T(Ax^{k+1}-b)+v^k.
\end{equation}
since $\|\xi^k\|\le \epsilon_k$ and $\epsilon_k \to 0$, we have $\|\xi^k\| \to 0$. We next consider two separate cases to complete the rest of the proof.

Case (a): $\|Ax^*-b\| < \sigma$. This together with $x^k \to x^*$ implies 
that $\|Ax^k-b\| \le \hat\sigma$ for some  $\hat\sigma<\sigma$ and 
all sufficiently large $k$. It then follows from this, \eqref{lag-penn} and 
Step 2) of Algorithm 3 that for all sufficiently large $k$,
\[
\mu_{k+1}=[\mu_k+\varrho_k(\|Ax^{k+1}-b\|^2-\sigma^2)]_+ = \varrho_k [\varrho^{-1}_k\mu_k+\|Ax^{k+1}-b\|^2-\sigma^2]_+ = 0.
\]
This together with \eqref{linear systemn} and $\xi^k \to 0$ implies $v^k \to 0$. It follows from this, $v^k \in \partial \Phi(x^{k+1})$, $x^k\to x^*$ and 
the semi-continuity of $\partial \Phi$ that $0\in\partial \Phi(x^*)$. Hence, \eqref{KKT} holds with $\mu^*=0$. 

Case (b): $\|Ax^*-b\| = \sigma$. By this and  the assumption, we know  $A^T(Ax^*-b)\neq0$. Claim that 
 $\{\mu_{k+1}\}$ is bounded. Suppose for contradiction that $\{\mu_{k+1}\}$ is unbounded. Assume without loss of generality that 
 $\mu_{k+1} \to \infty$. Recall that $x^k \to x^*$ and $\xi^k \to 0$. Dividing both sides of  
 \eqref{linear systemn} by $\mu_{k+1}$ and taking limits as $k\to\infty$, we obtain 
\[
\lim\limits_{k\to\infty} v^k/\mu_{k+1}= -2A^T(Ax^*-b).
\]
This together with $v^k\in\partial \Phi(x^k)$, $x^k\to x^*$ and $\mu_{k+1}\to\infty$ implies 
$-2A^T(Ax^*-b)\in\Phi^\infty(x^*)$ and hence $-2A^T(Ax^*-b)\in\Phi^\infty(x^*)\cap\cC$, which 
contradicts the assumption $\Phi^\infty(x^*)\cap\cC=\{0\}$ due to $A^T(Ax^*-b)\neq0$. Therefore,   
 $\{\mu_{k+1}\}$ is bounded. This together with \eqref{linear systemn} and 
 the boundedness of $\{x^k\}$ and $\{\xi^k\}$ implies that $\{v^k\}$ is bounded. Considering 
 a convergent subsequence if necessary, assume for convenience that 
  $(\mu_{k+1},v^k) \to (\mu^*,v^*)$. Using this, $x^k\to x^*$, $\xi^k\to 0$, and 
  taking limits on both sides of \eqref{linear systemn} as $k\to\infty$, we obtain that 
 \[
 2\mu^*A^T(Ax^*-b) + v^* = 0.
 \]
 Notice $v^*\in\partial\Phi(x^*)$. Hence, the first relation of \eqref{KKT} holds with such $\mu^*$. 
 In addition, the other two relations of \eqref{KKT} hold due to 
the assumption $\|Ax^*-b\| = \sigma$. This completes the proof.
\end{proof}

\section{Numerical results}
\label{results}

In this section we conduct numerical experiments to test the performance 
of the partially regularized models studied in this paper. Due to the paper 
length limitation, we only present the computational results for the partial $\ell_1$ regularized model, that is, $\phi(\cdot)=|\cdot|$. We also compare the performance of this model with some existing models. 
The codes for them are written in Matlab.  All computations are performed by Matlab R2015b running on a Dell desktop with a 3.40-GHz Intel Core i7-3770 processor and 16 GB of RAM.  

\subsection{Sparse logistic regression}\label{expmt1}

Sparse logistic regression has numerous applications in machine learning, computer vision, data mining, bioinformatics and signal processing 
(e.g., see \cite{Koh2007,LuZh13}). Given $m$ samples $\{a^1,\ldots, a^m\}$ with $n$ features, and $m$ binary outcomes $b_1,\ldots,b_m$, the goal of sparse logistic regression is to seek a sparse vector to minimize the {\it average logistic loss} function that is defined as
\[
l_{\avg}(w):= \frac{1}{m}\sum\limits_{i=1}^m \log(1+\exp(-b_i w^T a^i)) \quad\quad w\in \Re^n.
\]
This problem can be formulated as 
\beq
\label{slr}
\min_w \{l_{\avg}(w):\|w\|_0 \leq r\}, 
\eeq
where $r \in [1,n]$ is some integer for controlling the sparsity of the solution. 
Given that problem \eqref{slr} is generally hard to solve, one popular approach for finding an approximate solution of  
\eqref{slr} is by solving the full $\ell_1$ regularized problem:
\beq\label{l1}
\min_w l_{\avg}(w)+\lambda \sum\limits_{i=1}^n|w|_i
\eeq
for some $\lambda > 0$ (e.g., see \cite{Koh2007}).  Since the full $\ell_1$ regularization is typically biased,  we are interested in finding an approximate solution of \eqref{slr} by solving the partial $\ell_1$ regularized model:
\beq\label{pl1}
\min_w l_{\avg}(w)+\hat{\lambda} \sum\limits_{i=r+1}^n|w|_{[i]} 
\eeq
for some $\hat{\lambda} > 0$. Our aim below is to compare the performance of 
models \eqref{l1} and \eqref{pl1}. In particular, we solve them with suitable $\lambda$ and $\hat \lambda$ to find a solution with the same cardinality and then compare their respective average logistic loss. 

We apply Algorithm 1 (the NPG method) presented in Subsection \ref{npg-method} to solve problems \eqref{l1} and \eqref{pl1}. In our implementation of Algorithm 1, we set $L_{\min}=10^{-8}$, 
$L_{\max}=10^{8}$, $\tau=2$, $c=10^{-4}$ and $N=5$. In addition, we choose the initial point to be zero and set $L_k^0$ according to \eqref{Lk0}. We use \eqref{dist-x} as the termination criterion for Algorithm 1 and set the associated accuracy parameter $\epsilon$ to be $10^{-5}$.

In the first experiment, we compare the solution quality of models \eqref{l1}  and \eqref{pl1} on the random data of six different sizes. 
 For each size, we generate 10 instances in which the samples $\{a^1,\ldots, a^m\}$ and their outcomes $b_1,\ldots,b_m$ are generated in the same manner as described in \cite{Koh2007}. In detail, for each instance we choose an equal number of positive and negative samples, that is, $m_+ = m_- = m/2$, where $m_+$ (resp., $m_-$) is the number of samples with outcome $1$ (resp., $-1$). The features of positive (resp., negative) samples are independent and identically distributed, drawn from a normal distribution $N(\mu,1)$, where $\mu$ is in turn drawn from a uniform distribution on $[0,1]$ (resp., $[-1,0]$). For each instance, we first apply Algorithm 1 to solve \eqref{l1} with four different values of $\lambda$, which are $0.5\lambda_{\max}$, $0.25\lambda_{\max}$, $0.1\lambda_{\max}$ and $0.01\lambda_{\max}$, where $\lambda_{\max} = \|\sum^m_{i=1} b_i a^i\|_{\infty}/(2m)$. It is not hard to verify that $w=0$ is an optimal solution of \eqref{l1} for all $\lambda \ge \lambda_{\max}$ and thus $\lambda_{\max}$ is the upper bound on the useful range of $\lambda$.  For each $\lambda$, let $K$ be the cardinality of the approximate solution of \eqref{l1} found by Algorithm 1. Then we apply Algorithm 1 to solve \eqref{pl1} respectively with $r = \lceil 0.1K \rceil,\ \lceil 0.2K \rceil,\ \ldots,\ \lceil 0.9K \rceil,\ K$ and some $\hat{\lambda}$ (depending on $r$) so that the resulting approximate solution has cardinality no greater than $K$. 
The average $l_{\avg}$ over 10 instances of each size for different  $\lambda$ and $r$ is presented in Figure \ref{f1}. The result of model \eqref{l1} corresponds to the part of this figure with $r=0$. We can observe that model \eqref{pl1} with the aforementioned positive $r$ substantially outperforms model \eqref{l1} in terms of the solution quality since it generally achieves lower average logistic loss while the sparsity of their solutions is similar. In addition, the average logistic loss of for model \eqref{pl1} becomes smaller as $r$ gets closer to $K$, which indicates more alleviation on the bias of the solution.
%

\begin{figure}
\centering
\subfigure{
\includegraphics[scale=0.415]{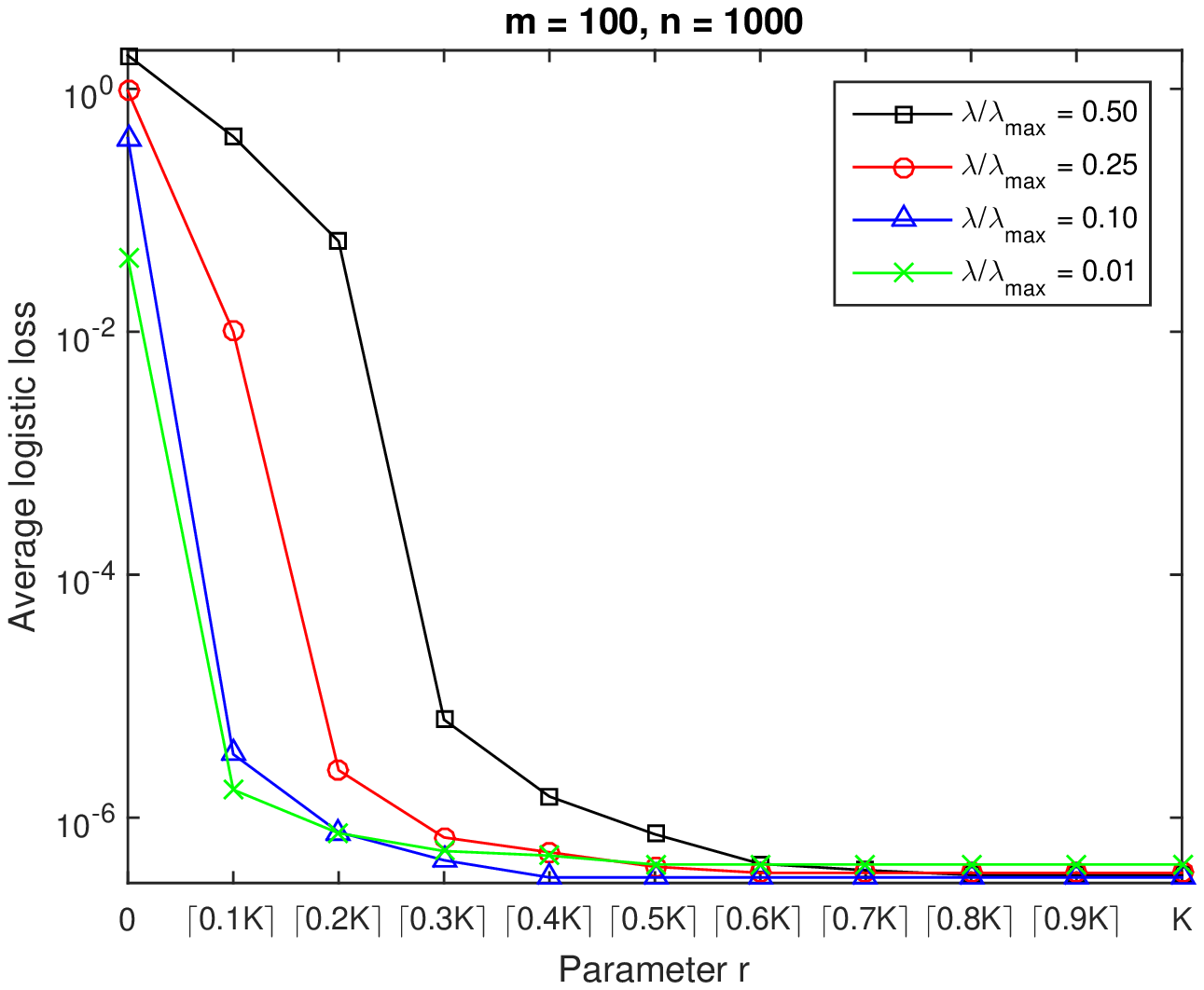}
\includegraphics[scale=0.415]{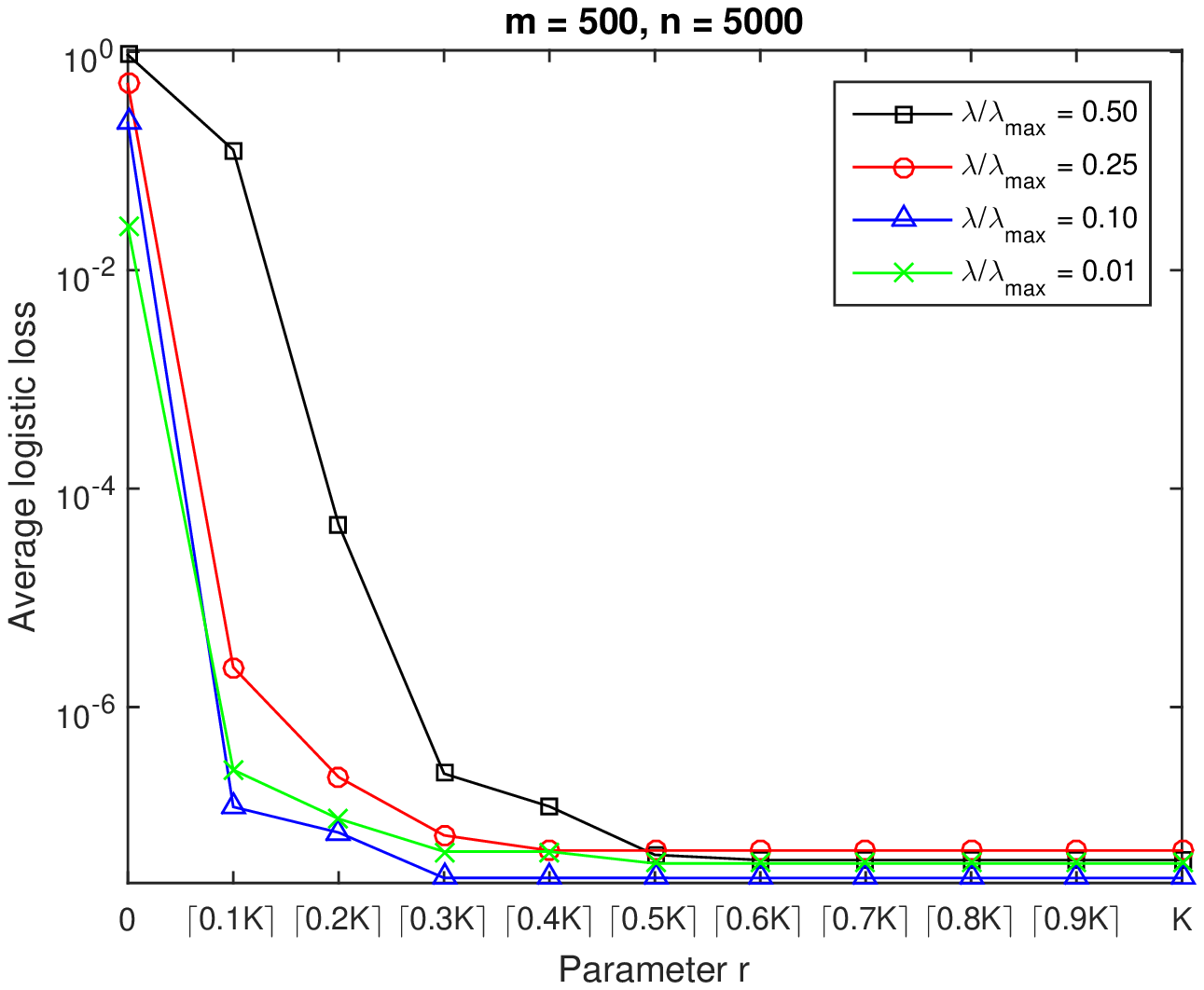}
}
\subfigure{
\includegraphics[scale=0.415]{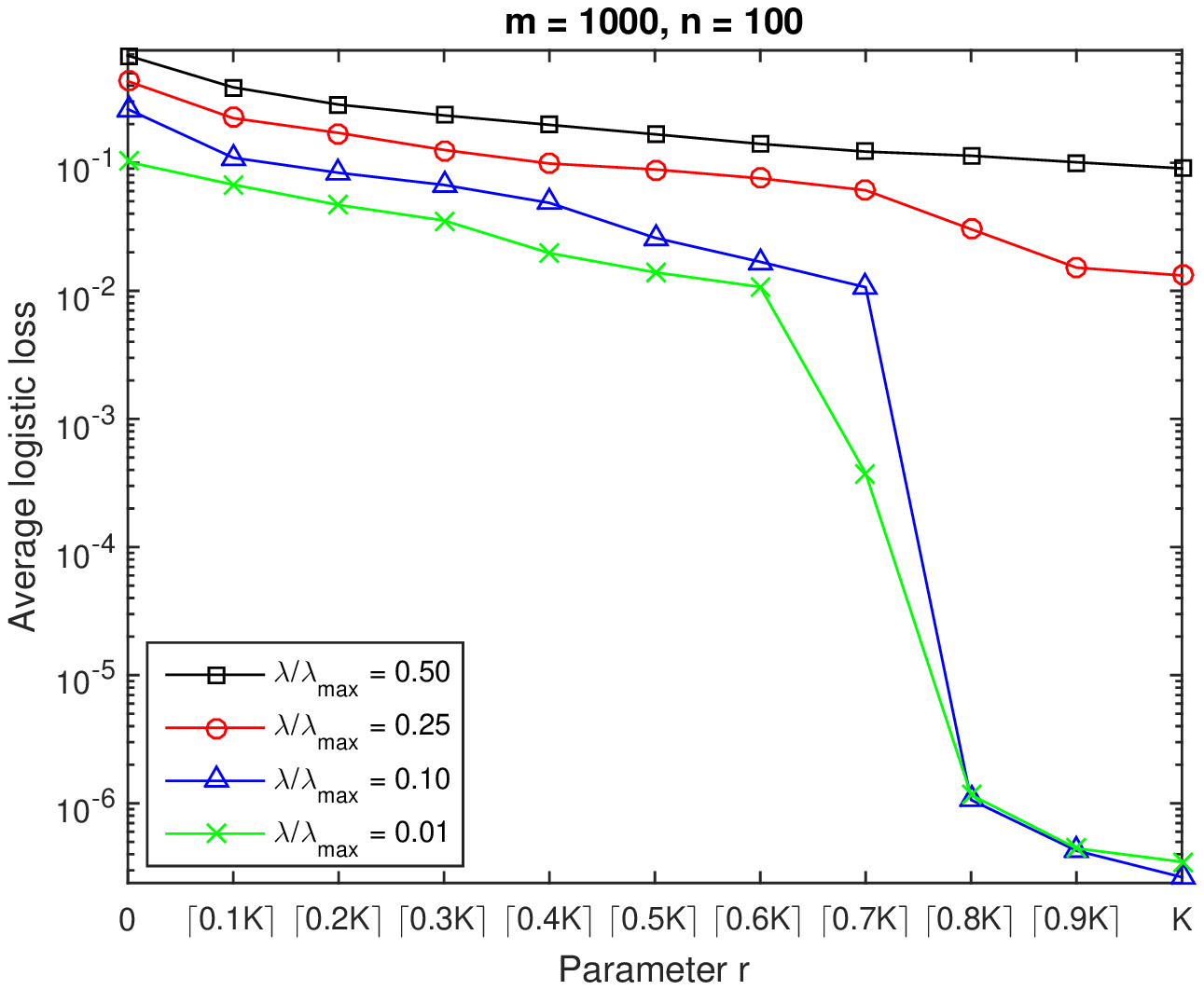}
\includegraphics[scale=0.415]{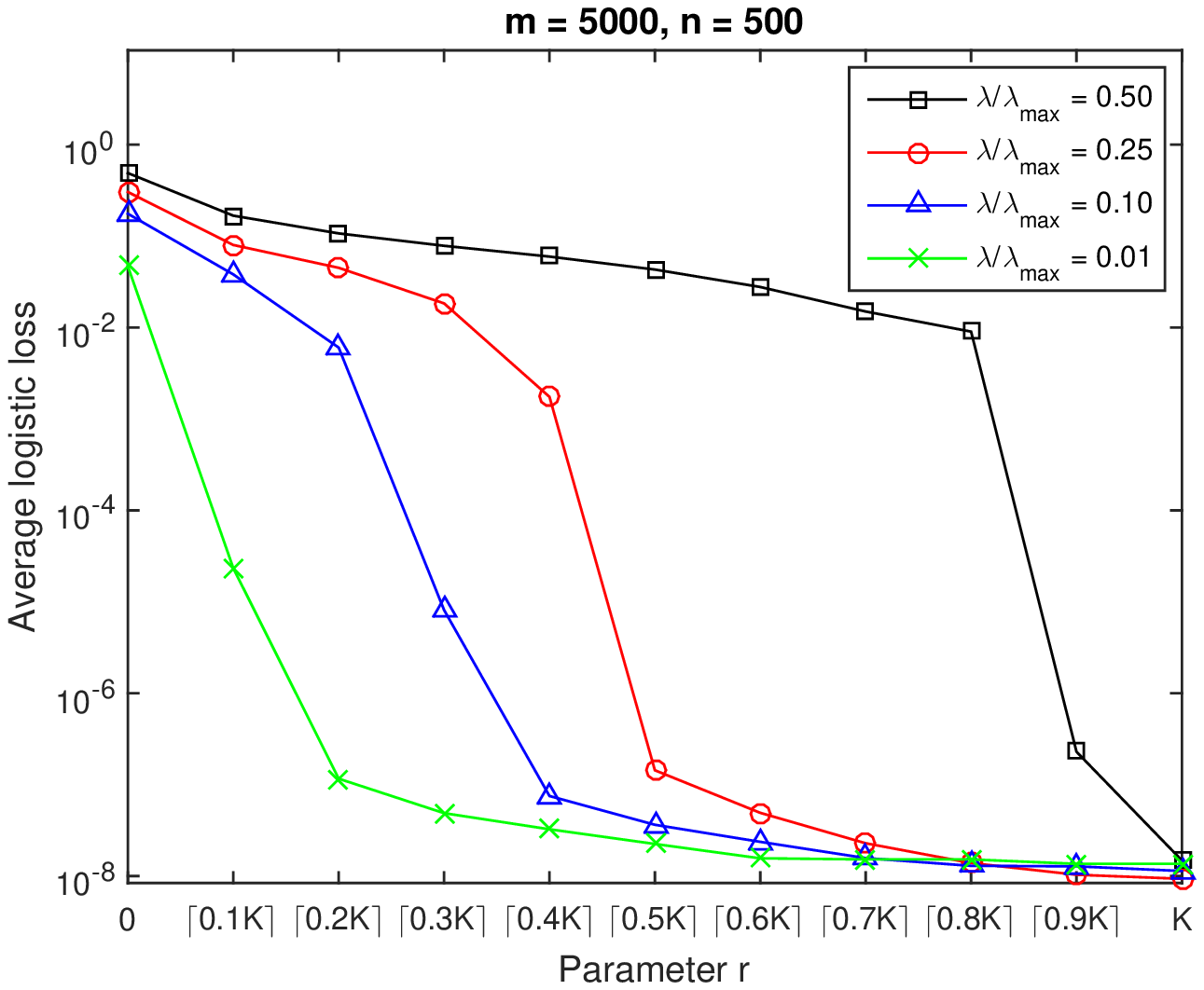}
}
\subfigure{
\includegraphics[scale=0.415]{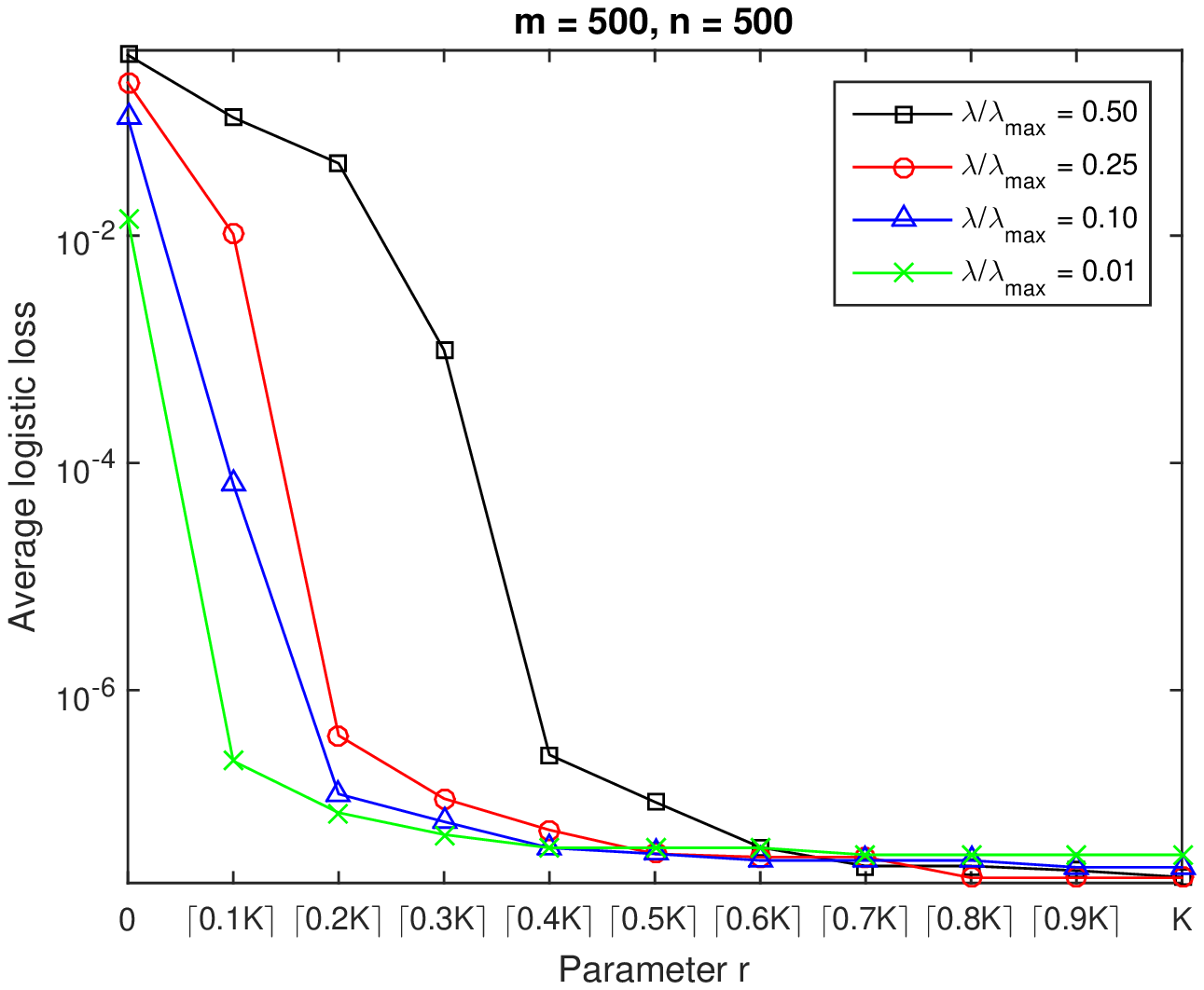}
\includegraphics[scale=0.415]{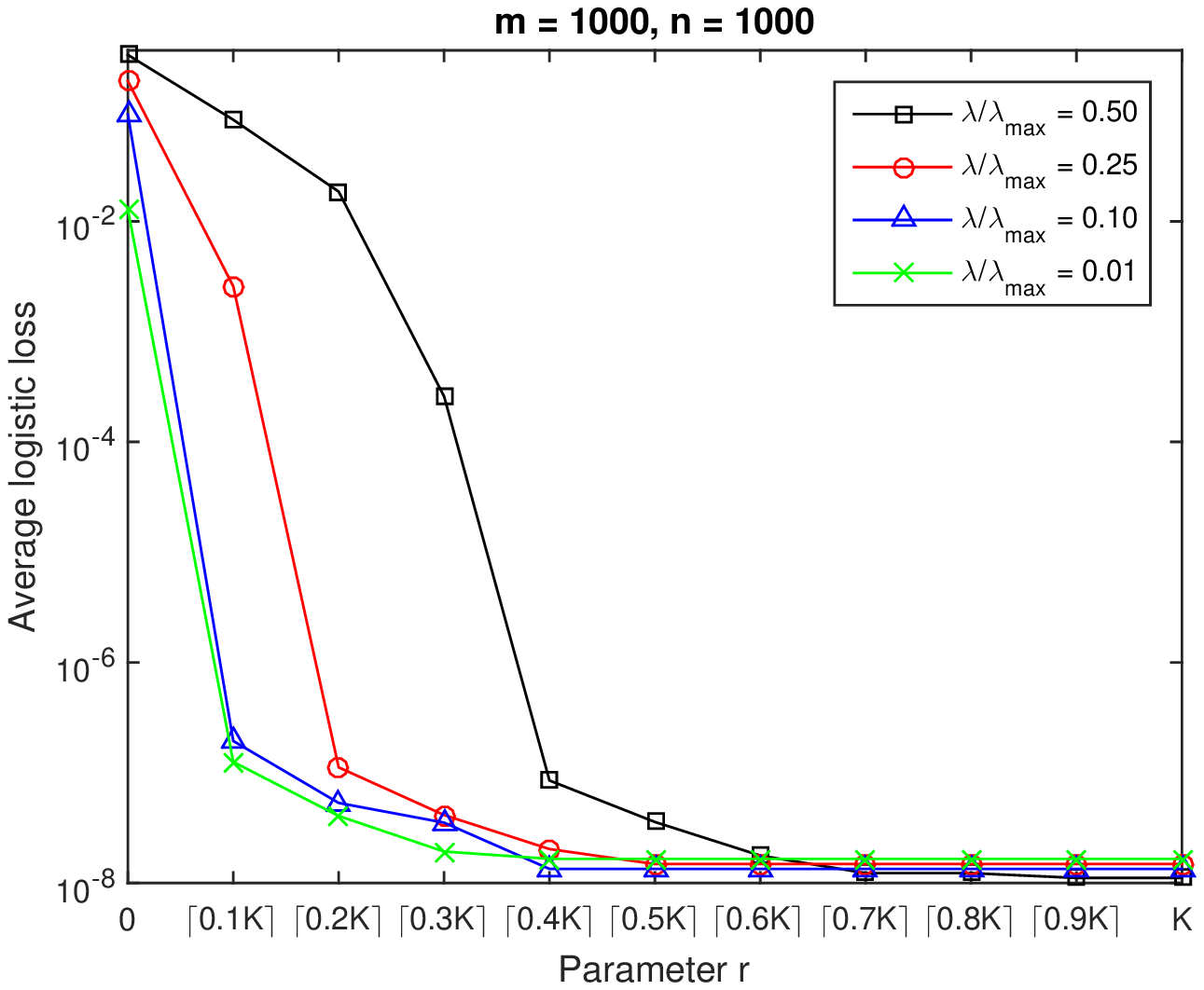}
}
\caption{Computational results on six random problems}
\label{f1}
\end{figure}

In the second experiment, we compare the solution quality of models \eqref{l1} and \eqref{pl1} on three small- or medium-sized benchmark data sets which are from the UCI machine learning bench market repository \cite{Lichman2013} and other sources \cite{Golub1996}. The first data set is the colon tumor gene expression data \cite{Golub1996} with more features than samples; the second one is the ionosphere data \cite{Lichman2013} with fewer features than samples; and the third one is the musk data \cite{Lichman2013} where we take as many samples as features. We discard the samples with missing data, and standardize each data set so that the sample mean is zero and the sample variance is one. We repeat the above experiment for these data sets and present the computational results in Figure \ref{f2}. One can see that the solution of model \eqref{pl1} achieves lower average logistic loss than that of model \eqref{l1} while they have same cardinality. In addition, the average logistic loss of model \eqref{pl1} decreases as $r$ increases.

\begin{figure}[t!]
\centering
\subfigure{
\includegraphics[scale=0.415]{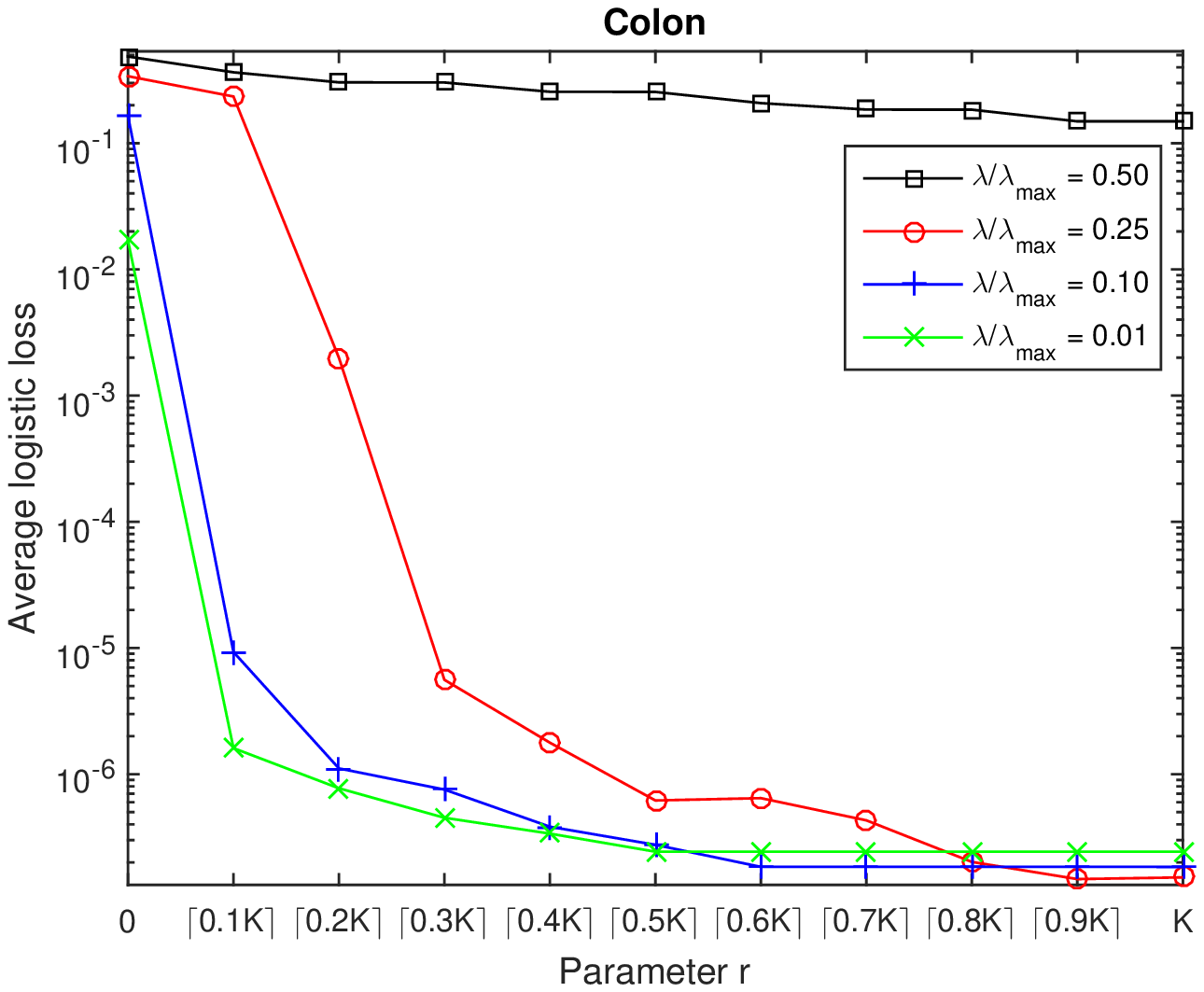}
}
\subfigure{
\includegraphics[scale=0.415]{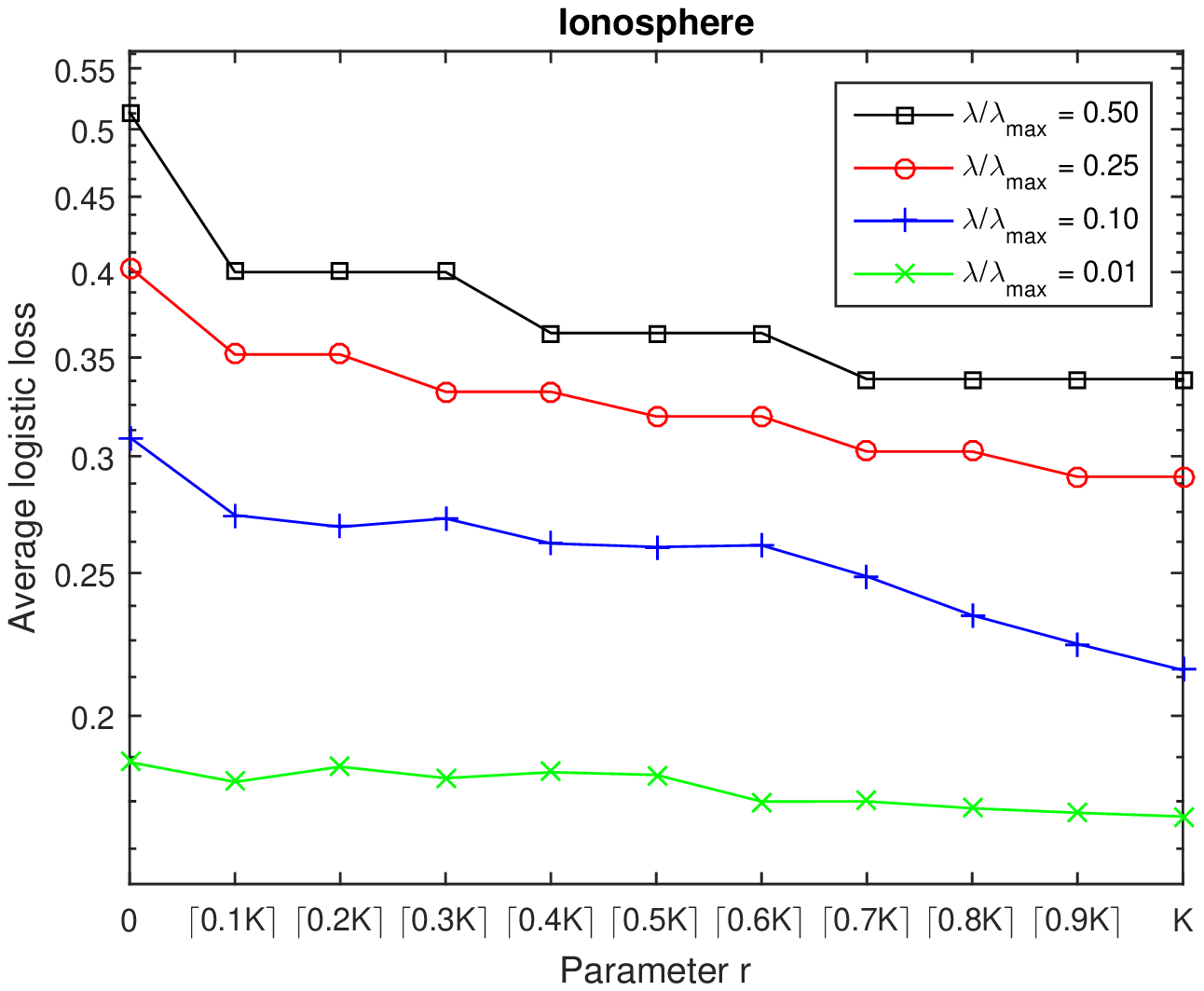}
\includegraphics[scale=0.415]{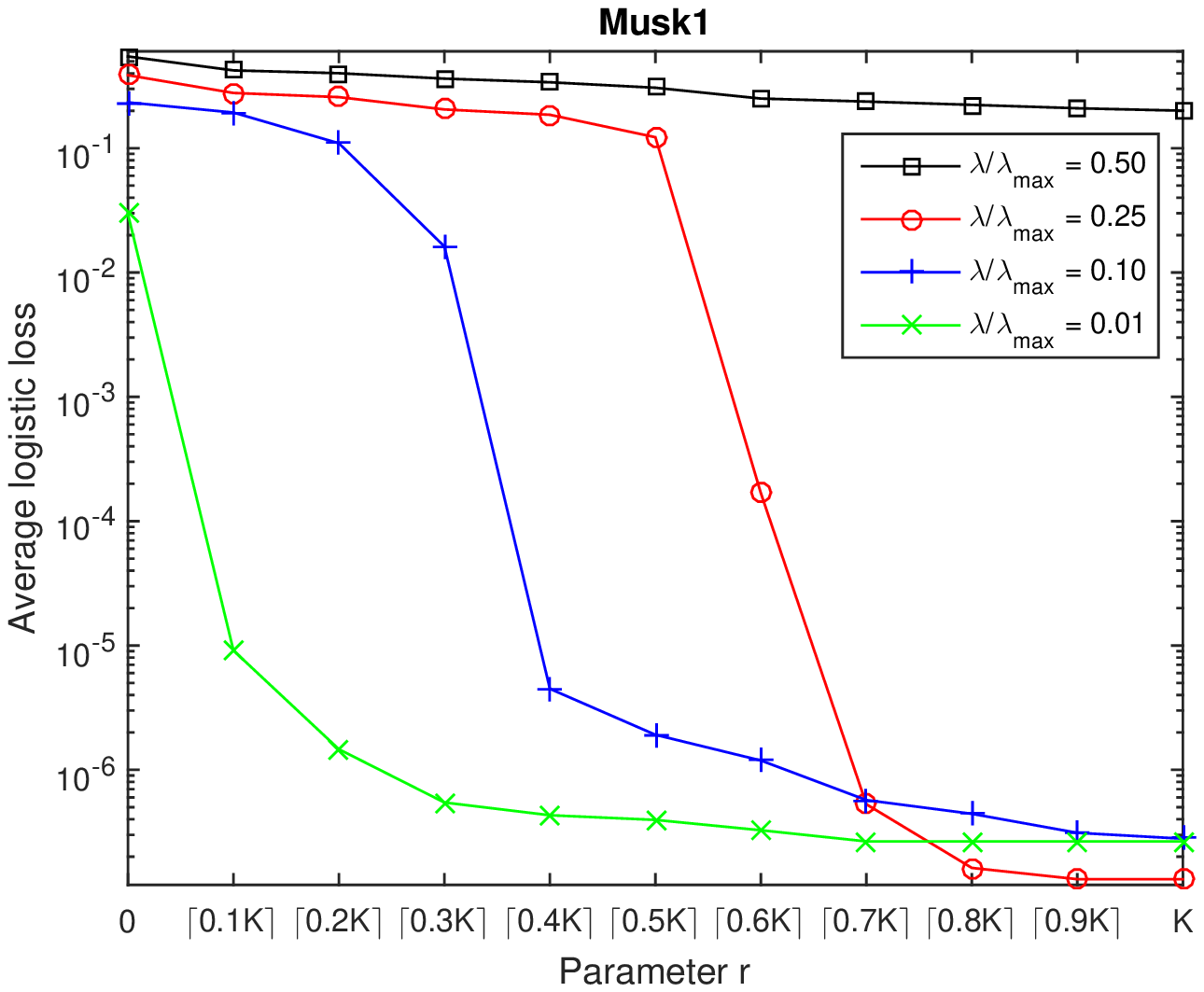}
}
\caption{Computational results on three real data sets}
\label{f2}
\end{figure}

\subsection{Sparse recovery for a linear system}

In this subsection we consider finding a sparse solution to a linear system, which can be formulated as
\beq \label{lsl0}
\min_x \left\{ \|x\|_0: \|Ax-b\|^2 \leq \sigma^2 \right\},
\eeq
where $A \in \Re^{m \times n}$, $b\in \Re^n$ and $\sigma \geq 0$ is 
the noise level (that is, $\sigma=0$ for the noiseless case and $\sigma>0$ for the noisy case). 
Given that problem \eqref{lsl0} is typically hard to solve, one common approach in the literature for finding an approximate solution of \eqref{lsl0} is by solving the model:
\beq \label{Lp}
\min_{x} \left\{ \sum\limits_{i=1}^n \phi(|x|_i): \|Ax-b\|^2 \leq \sigma^2 \right\}
\eeq
for some regularizer $\phi:\Re_+ \to \Re_+$. 
Our aim below is to compare the performance of model \eqref{Lp} with the partial $\ell_1$ regularized model
\beq \label{lspl1}
\min_x \left\{ \sum\limits_{i=r+1}^n|x|_{[i]}: \|Ax-b\|^2 \leq \sigma^2 \right\}
\eeq
for some integer $r>0$. In our comparison, the particular $\phi$ for \eqref{Lp} is chosen to be the  regularizers $\ell_1$, $\ell_q$, Log, Capped-$\ell_1$, MCP and SCAD, which are presented in Section \ref{intro}. The associated parameters for them are chosen as  $q=0.5$, $\varepsilon=10^{-3}$,  $\nu=10^{-2}$, $\lambda=1$, $\alpha=2.7$ and $\beta=3.7$ that are commonly used in the literature. 

We apply FAL methods (Algorithms 2 and 3) proposed in Subsections \ref{noiseless-sys} and \ref{noisy-sys}  to solve  both \eqref{Lp} and \eqref{lspl1} with $\sigma=0$ and $\sigma>0$, respectively. 
We now address the initialization and termination criterion for Algorithms 2 and 3. In particular, we set $x^\feas=A\backslash b$, $x^0=0$, $\mu^0=0$, $\mu_0=0$, $\rho_0=1$, $\gamma=5$, $\eta=0.25$, $\theta=10^{-2}$,  $\epsilon_0=1$ and $\epsilon_k = \max\{0.1\epsilon_{k-1},10^{-5}\}$ for $k\ge 1$. In addition, we choose 
the termination criterion $\|Ax^k-b\|_{\infty} \le 10^{-5}$ and $\epsilon_k \le 10^{-4}$  for Algorithm 2, and $[\|Ax^k-b\| _{\infty}-\sigma]_+ \le 10^{-5}$ and $\epsilon_k \le 10^{-4}$ for Algorithm 3. The augmented Lagrangian subproblems arising in FAL are solved by Algorithm 1, whose  parameter 
setting and termination criterion are the same as those specified in Subsection \ref{expmt1}. 

We next conduct numerical experiments to compare the performance of models \eqref{Lp} and \eqref{lspl1} for the aforementioned $\phi$. In the first experiment, we compare them on some noiseless random data that is generated 
in the same manner as described in \cite{FoLa09}.
We set $m=128$ and $n=512$. For each $K\in \{1,2,\ldots,64\}$,  we first randomly generate $100$ pairs of $K$-sparse vector $x^*\in\Re^n$ and  $m \times n$ data matrix $A$. In particular, we randomly choose $K$ numbers from $\{1,2,\ldots,512\}$ as the support for $x^*$ and then generate the nonzero entries of $x^*$ and matrix $A$ according to the standard Gaussian distribution. Finally we orthonormalize the rows of $A$ and set $b=Ax^*$.  We apply Algorithm 2 to solve models \eqref{Lp} and \eqref{lspl1} with $r = \lceil 0.1K \rceil,\ \lceil 0.2K \rceil,\ \ldots,\ \lceil 0.9K \rceil,\ K$, $\sigma=0$ and the aforementioned $\phi$. It was observed in our experiment that all models can exactly recover $x^*$ for $K\in \{1,2,\ldots,19\}$. Therefore, we will not report the results for these $K$. To evaluate the solution quality for the rest of $K$, we adopt a criterion as described in 
\cite{CaWaBo08}.
We say that a sparse vector $x$ is {\it successfully recovered} by an estimate $\hat{x}$ if $\|x-\hat{x}\|< 10^{-3}$. 
The computational results for the instances generated above are plotted in Figure \ref{f4}. In detail, the average frequency of success over all instances for model \eqref{lspl1} with  $r = \lceil 0.1K \rceil,\ \lceil 0.2K \rceil,\ \ldots,\ \lceil 0.9K \rceil,\ K$ against $K\in \{20,21,\ldots,64\}$ is presented in the first subfigure.  The average frequency of success of model \eqref{lspl1} with  $r = K$ and model \eqref{Lp} with the aforementioned $\phi$ against $K\in \{20,21,\ldots,64\}$ is plotted in the second subfigure, where the partial $\ell_1$ represents model \eqref{lspl1} with  $r = K$ and the others represent model \eqref{Lp} with various $\phi$, respectively. Analogously, the accumulated CPU time over all instances for each model against $K\in \{20,21,\ldots,64\}$ is displayed in the last subfigure. One can observe from the first subfigure that for each $K$, the average frequency of success of model \eqref{lspl1} becomes higher as $r$ increases. In addition, from the last two subfigures, we can see that for each $K\in \{20,21,\ldots,64\}$, model \eqref{lspl1} with $r=K$ generally has higher average frequency of success than and comparable average CPU time to model \eqref{Lp} with various $\phi$.


\begin{figure}[t!]
\centering
\subfigure[Frequency of success for model \eqref{lspl1}]{
\includegraphics[scale=0.408]{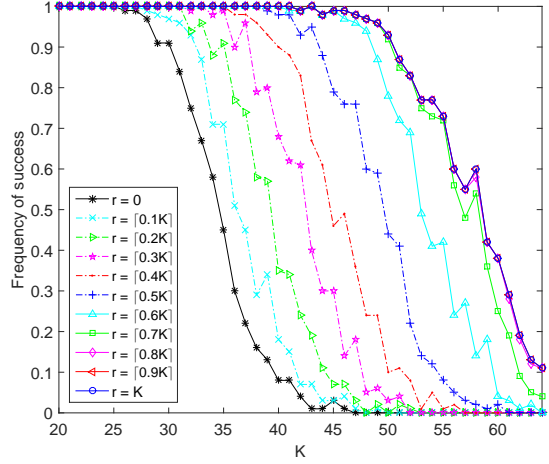}}
\\
\subfigure[Frequency of success for models \eqref{Lp} and \eqref{lspl1}]{
\includegraphics[scale=0.408]{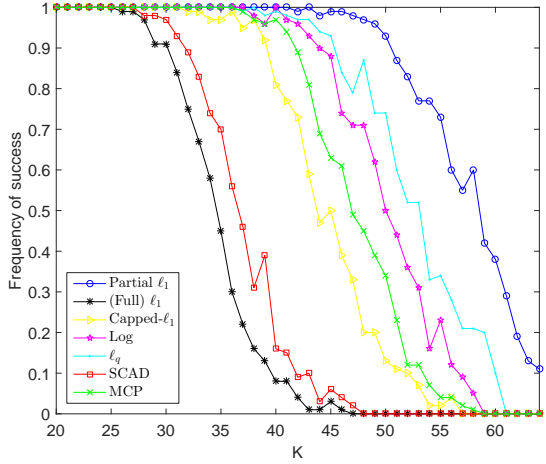}}
\subfigure[Average CPU time for models \eqref{Lp} and \eqref{lspl1}]{
\includegraphics[scale=0.408]{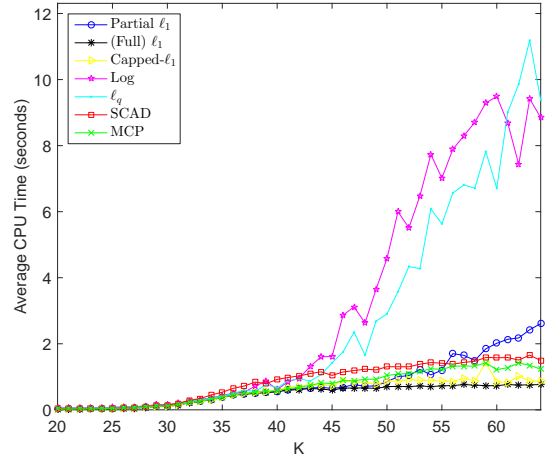}}
\caption{Computational results on random noiseless data}
\label{f4}
\end{figure}

In the second experiment, we compare the performance of models \eqref{Lp} and \eqref{lspl1} with the aforementioned $\phi$ on some noisy random data. For each $K\in \{1,2,\ldots,64\}$,  we generate $100$ pairs of $K$-sparse vector $x^*\in\Re^n$ and $m \times n$ data matrix $A$ in the same manner as above. For each pair of $x^*$ and $A$, we set $b=Ax^*+\xi$ and $\sigma=\|\xi\|$, where $\xi$ is drawn from a normal distribution with mean $0$ and variance $10^{-4}$. For each such instance, we apply Algorithm 3 to solve the models \eqref{Lp} and \eqref{lspl1} and compute their corresponding relative error according to $\text{rel\_err}=\|\hat{x}-x\|/\|x\|$, which evaluates how well a sparse vector $x$ is recovered by an estimate $\hat{x}$ 
as described in \cite{CaWaBo08}. The computational results of this experiment are presented in Figure \ref{f5}.
We plot the average relative error over all instances for model \eqref{lspl1} with $r = \lceil 0.1K \rceil,\ \lceil 0.2K \rceil,\ \ldots,\ \lceil 0.9K \rceil,\ K$ against $K$ in the first subfigure. The average relative error of model \eqref{lspl1} with $r=K$ and \eqref{Lp} with the aforementioned $\phi$ against $K$ is displayed in the second  subfigure, where the partial $\ell_1$ represents model \eqref{lspl1} with  $r = K$ and the others represent model \eqref{Lp} with various $\phi$, respectively. Similarly, the accumulated CPU time over all instances for each model against $K$ is presented in the last subfigure. One can see from the first subfigure that for each $K$, the average relative error of model \eqref{lspl1} becomes lower as $r$ increases. From the second subfigure, we can observe that when $0<K<38$, the average relative error of  model \eqref{lspl1}  with $r=K$ is comparable to that of model \eqref{Lp} with $\phi$ chosen as $\ell_q$, Log, and Capped-$\ell_1$,  while it is lower than the one with $\phi$ chosen as $\ell_1$, SCAD and MCP.  Nevertheless,  when $38\leq K \leq 64$, model \eqref{lspl1}  with $r=K$ has lower 
average relative error than the other models. In addition, as seen from the last subfigure, model \eqref{lspl1} with $r=K$ generally has comparable  average CPU time to model \eqref{Lp} with various $\phi$.

\begin{figure}[t!]
\centering
\subfigure[Average relative error for model \eqref{lspl1}]{
\includegraphics[scale=0.408]{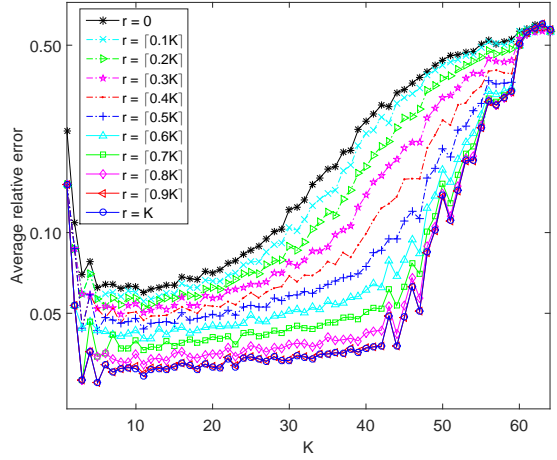}}
\\
\subfigure[Average relative error for models \eqref{Lp} and \eqref{lspl1}]{
\includegraphics[scale=0.408]{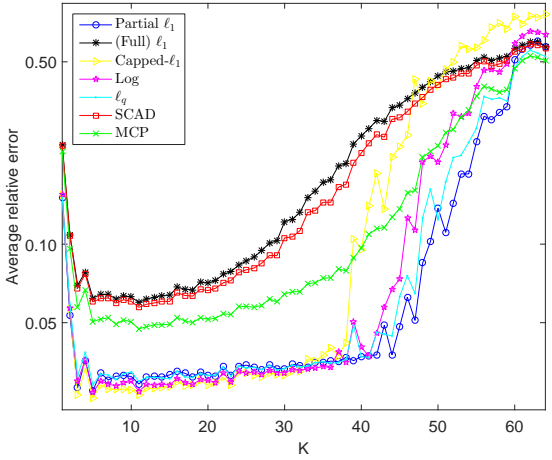}}
\subfigure[Average CPU time for models \eqref{Lp} and \eqref{lspl1}]{
\includegraphics[scale=0.408]{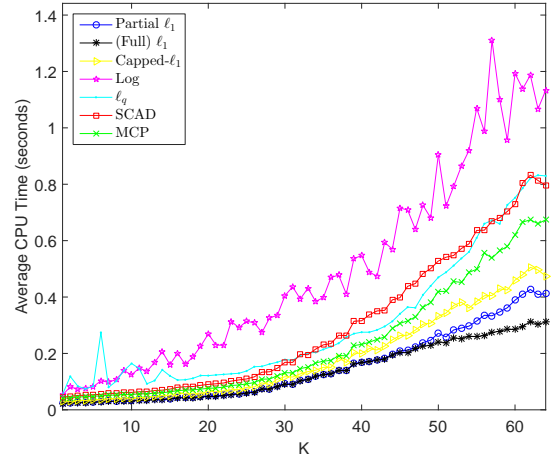}}
\caption{Computational results on random noisy data}
\label{f5}
\end{figure}

\section{Concluding remarks}
\label{remarks}

In this paper we proposed a class of models with partial regularization for recovering a sparse solution of a linear system. We studied some theoretical properties for these models including sparsity inducing, local or global recovery and also stable recovery. 
We also developed an efficient first-order feasible augmented Lagrangian method for solving them whose subproblems are solved by a nonmonotone proximal gradient 
method. The global convergence of these methods were also established. Numerical results on compressed sensing and sparse logistic regression demonstrate that our proposed models substantially outperform the widely used ones in the literature in terms of solution quality.
      
It shall be mentioned that the partially regularized model has recently been proposed in 
\cite{Gao10,GaSu10} for finding a low-rank matrix lying in a certain convex subset of positive semidefinite cone. In particular, the function $\sum^n_{i=r+1} \lambda_i(X)$ was used       
as a partial regularizer there, where $0 <r <n$ and $\lambda_i(X)$ is the $i$th largest 
eigenvalue of an $n \times n$ positive semidefinite matrix $X$. Clearly, this regularizer can be generalized to a class of partial regularizers in the form of $\sum^l_{i=r+1}\phi(\sigma_i(X))$, where $\sigma_i(X)$ is the $i$th largest singular value of a 
$m \times n$ matrix $X$, $l=\min\{m,n\}$ and $\phi$ satisfies Assumption \ref{assump-phi}. Most results of this paper can be generalized to the following model for finding a low-rank matrix:
\[
\min\limits_{X\in\Re^{m\times n}} \left\{\sum^l_{i=r+1}\phi(\sigma_i(X)): \|{\cal A}(X)-b\| \le \sigma \right\},
\]
where ${\cal A}: \Re^{m\times n} \to \Re^p$ is a linear operator, $b\in\Re^p$ and $\sigma \ge 0$. 
This is beyond this paper and left as a future research.

\end{document}